\documentclass[a4paper, 12pt]{amsart}

\usepackage{amsmath}
\usepackage{amssymb}
\usepackage{amsthm}
\usepackage[usenames,dvipsnames]{xcolor}

\usepackage[final]{showkeys}

\usepackage{url}
\usepackage{ mathrsfs }

\definecolor{DeepBlue}{rgb}{0,0,.7}
\usepackage[dvipdfm, setpagesize=false]{hyperref}
\hypersetup{colorlinks=true, citecolor=Sepia, linkcolor=DeepBlue, urlcolor=DeepBlue, pdfborder= 0 0 0}
\hypersetup{setpagesize=false}
\usepackage[all]{hypcap}
\usepackage[all]{xy}

\newcommand{\HH}{\mathbb{H}}

\newcommand{\rk}{\mathrm{rk}}

\newcommand{\R}{\mathbb{R}}
\newcommand{\Ob}{\mathcal{O}}
\newcommand{\C}{\mathbb{C}}
\newcommand{\Z}{\mathbb{Z}}

\newcommand{\dt}{\frac{\partial}{\partial t}}

\newcommand{\id}{\mathrm{id}}

\newcommand{\bs}{\backslash}
\newcommand{\D}{\Delta}
\newcommand{\al}{\alpha}
\newcommand{\Id}{\text{\textnormal{Id}}}
\newcommand{\End}{\text{End}}
\newcommand{\tr}{\text{\textnormal{tr}}\thinspace}
\newcommand{\Tr}{\text{\textnormal{Tr}}\thinspace}
\newcommand{\Trs}{\text{\textnormal{Tr}}_s\thinspace}
\newcommand{\spec}{\text{\textnormal{spec}}\thinspace}
\newcommand{\vol}{\text{\textnormal{vol}}}
\newcommand{\Cl}{\text{\textnormal{Cl}}}
\newcommand{\ad}{\text{\textnormal{ad}}}
\newcommand{\supp}{\text{\textnormal{supp}}}
\newcommand{\Ad}{\text{\textnormal{Ad}}}
\newcommand{\RRe}{\text{\textnormal{Re}}}
\newcommand{\IIm}{\text{\textnormal{Im}}}
\newcommand{\T}{\mathbb{T}}
\newcommand{\N}{\mathbb{N}}

\usepackage{accents}
\newlength{\dtildeheight}

\newtheorem{Th}{Theorem}[section]
\newtheorem{Lemma}[Th]{Lemma}
\newtheorem{Proposition}[Th]{Proposition}

\theoremstyle{definition}
\newtheorem{remark}{Remark}[section]
\newtheorem{definition}{Definition}[section]

\newtheorem{exmp}{Example}[section]

\begin{document}

\title{The twisted Selberg trace formula and the Selberg zeta function for compact orbifolds}
\author{Ksenia Fedosova}
\address{Max Planck Institute for Mathematics\\ Vivatgasse 7\\ 53111 Bonn \\ Germany}

\email{fedosova@math.uni-bonn.de}

\begin{abstract}
We study elements of the spectral theory of compact hyperbolic orbifolds $\Gamma \bs \HH^{n}$. We establish a version of the Selberg trace formula for non-unitary representations of $\Gamma$ and prove that the associated Selberg zeta function admits a meromorphic continuation to $\C$.
\end{abstract}

\maketitle

\setcounter{tocdepth}{1}

\section{Introduction}
\label{SecInt}

The Selberg zeta function is an important tool in the study of the spectral theory of  locally symmetric Riemannian spaces. This zeta function is defined by an infinite product over the closed geodesics that only converges in a complex half-space, so for its investigation it is useful to understand if it admits a meromorphic continuation.
The purpose of this paper is to prove the existence of a meromorphic continuation
of the Selberg zeta functions on compact orbifolds by establishing a suitable Selberg trace formula.

\begin{Th}
	\label{THEOREM1}
	Suppose $\Ob = \Gamma \bs \HH^{2n+1}$ is an compact odd-dimensional hyperbolic 
	 orbifold, $\chi$ is a (possibly) non-unitary representation of $\Gamma$, and $\sigma$
	is a unitary representation of $SO(2n)$. Then the Selberg zeta function $Z(s, \sigma, \chi)$ (see Definition \ref{thermos})
	admits a meromorphic continuation to $\C$.
\end{Th}

To prove Theorem \ref{THEOREM1} it is sufficient to show that the residues of $Z'(s,\sigma, \chi)/Z(s,\sigma, \chi)$
are integers. This was proven by \cite{BO} in the case when $\Ob$ is a compact hyperbolic manifold and $\chi$ is  unitary. 
Later on the theorem was extended to non-compact finite volume hyperbolic manifolds with cusps in the case when  $\chi$ is unitary \cite{Par} and  when $\chi$ is a restriction of a representation of $SO_0(2n+1,1)$ \cite{Pf}. Using a slightly different approach, the theorem was proved in \cite{Tsu} for  compact orbifolds when $\chi$ and $\sigma$
are trivial representations.
Notably, the theorem does not necessarily hold for non-compact finite volume hyperbolic orbifolds: an example is  the Bianchi orbifold of discriminant~$-3$ with $\chi$ and $\sigma$ trivial  \cite{Fri}.

The approach of \cite{BO,Pf,Par} invokes  applying the Selberg trace formula to a certain test function which makes  $Z'(s,\sigma, \chi)/Z(s,\sigma, \chi)$  appear as one of the terms in the geometric side of the formula. In order to adopt their approach we need to prove a more general version of the Selberg trace formula.

\begin{Th}\label{T2}
	Let $G$ be a connected real semisimple Lie group with finite center of non-compact type,
	$K$ a maximal compact subgroup of $G$, and $\Gamma \subset G$ a discrete subgroup such that
	$\Ob := \Gamma \bs G / K$ is compact. Let $\chi$ be a non-unitary representation of $\Gamma$, and $\sigma$ be 
	a unitary representation of $K$. For the non-selfadjoint Laplacian $\Delta_{\chi, \sigma}^\#$ defined in Subsection \ref{Sec8}  and $\phi$ belongs to the space $\mathcal{P}(\C)$ defined in Section \ref{sec:!fa},
	$$ \sum_{\lambda \in \spec(\Delta_{\chi, \sigma}^\#)} m(\lambda) \phi(\lambda^{1/2}) = \sum_{\{\gamma\} \subset \Gamma} \vol(\Gamma_\gamma \bs G_\gamma) I_\phi(\gamma).$$
	Above $m(\lambda)$ is the multiplicity of $\lambda$; $\{\gamma\}$ denotes the conjugacy class of $\gamma$; $G_\gamma$ and $\Gamma_\gamma$ are the centralizers of $\gamma$ in $G$ and $\Gamma$, respectively. Finally, $I_\phi(\gamma)$ are the so-called orbital integrals, defined by
	$$ I_\phi(\gamma) := \int_{G_\gamma \bs G} \tr h_\phi(g \gamma g^{-1}) d \dot{g},$$
	where $h_\phi$ is the integral kernel of the self-adjoint Laplacian $\widetilde{\Delta}_\sigma$ defined in 
	Section 5.
\end{Th}
The Selberg trace formula has a rich history starting from the classical work  \cite{Sel}, but has mostly been constrained 
to unitary representations $\chi$ of $\Gamma$.
The non-unitary case was first studied in \cite{Mu} under the assumption that $\Gamma$ contains no elements of finite order, also called elliptic elements, which means $\Ob$ is a compact  manifold. 

To complete the proof of Theorem \ref{THEOREM1} we apply Theorem \ref{T2} 
to the case $G/K = \HH^{2n+1}$. The major remaining problem is to calculate the orbital integrals $I_\phi(\gamma)$ corresponding to elliptic elements $\gamma\in \Gamma$. 
\begin{Lemma}\label{LEMMA1}
	In the above setup the orbital integral $I_\phi(\gamma)$ for elliptic $\gamma \in \Gamma$ equals
	$$ I_\phi(\gamma) = \sum_{\sigma' \in \widehat{SO(2n)}} \int_\R \Theta_{\lambda, \sigma'}(\phi) P_\gamma(i \lambda) d \lambda,$$
	where $\widehat{SO(2n)}$ is the set of equivalence classes of irreducible representations of $SO(2n)$, $\Theta_{\sigma', \lambda}(\phi)$ is the character of the unitarily induced representation $\pi_{\sigma', \lambda}$ of $G$, and $P_\gamma(\lambda)$ is an even polynomial.
\end{Lemma}
Orbital integrals have so far been computed for $G/K = \HH^2$, $\HH^3$ and $\HH^{2n}$
in \cite{Gel}, \cite{Kna} and \cite{SW} respectively. The computation of orbital integrals is not only useful for the proof of Theorem~\ref{THEOREM1}, but also for other applications of the Selberg trace formula. For example, we will use Lemma~\ref{LEMMA1} in the upcoming papers \cite{Fe2, Fe3} to study the behavior of the 
analytic torsion of odd-dimensional compact and non-compact finite-volume orbifolds.

During the proof of Theorem \ref{T2} we obtain a result about the heat trace expansion of generalized Laplacians on orbibundles.
\begin{Lemma}\label{mmma}
	Let $E \to \Ob$ be an orbibundle over a compact Riemannian orbifold $\Ob$, and $\Delta$ the generalized Laplacian acting on sections of $E$. Then there exists an asymptotic expansion as $t \to 0$:
	 \begin{equation}\label{hte}
	 \Tr e^{-t \Delta} = \frac{1}{2 \pi t^{\dim(\Ob)/2}} \sum_{k=0}^\infty a_k t^{k/2},
	 \end{equation}
	 where $a_{k} \in \R$. The constant term in (\ref{hte}) vanishes
	 if the dimensions of $\Ob$ and all its singular strata are odd.
\end{Lemma}
Note that $\Ob$ is now not necessarily a global quotient hyperbolic orbifold.
Lemma \ref{mmma} generalizes \cite{Dry}, which established the heat trace expansion 
for the scalar Laplacian on orbifolds. We will use Lemma \ref{mmma} in the upcoming paper \cite{Fe2} to show that the analytic torsion on odd-dimensional orbifolds with odd-dimensional strata is independent of the metric.


\subsection*{Acknowledgement}
The present paper is a part of the author's PhD thesis. She would like to thank her supervisor Werner M\"uller for his constant support, Dmitry Tonkonog and Julie Rowlett for reading the draft and correcting minor mistakes, and Jonathan Pfaff for useful discussions. 

\section{Orbifolds, orbibundles and pseudodifferential operators}\label{sec:orb}

\subsection{Definitions and examples}
We begin with an informal introduction to orbifolds and orbibundles.
\begin{definition}
An orbifold $\Ob$ is a topological space such that for each $p \in \Ob$ there exists a neighbourhood $U_p$, an open contractible set
 $V_p$ on $\R^n$, and a finite group 
$\Gamma_p$ acting on $V_p$ and satisfying $U_p = V_p / \Gamma_p$.
\end{definition} 
\begin{definition}
An orbifold is said to be global quotient if it is the orbit space of a manifold under a global action of a discrete (not necessarily finite) group.
\end{definition}
\begin{definition}
An orbibundle $\pi: E \to \Ob$  is a map with the following properties. 
Let $p \in \Ob$ and $U_p, \Gamma_p, V_p$ be as above. Then there exists a representation 
of $\Gamma_p$ on $\R^k$ such that the restriction of $\pi$ to $\pi^{-1}(U_p)$ is diffeomorphic to 
$\pi'=(V_p \times \R^k)/\Gamma_p\to V_p/\Gamma_p$, where the action of $\Gamma_p$ on $V_p \times \R^k$ is the diagonal one.
\end{definition}
We will not recall the standard cocycle condition which ensures the local data on different neighborhoods can be glued well together.

\begin{exmp}\label{ex1}
Let $U_p = B$ be a unit ball in $\R^2$ with centre $p=0$, $\R^k = \R$, $\Gamma_p = \Z_2 = \{ e,i\}$. 
with the non-trivial element of $\Z_2$ acting by $i(x) = -x$. 
Then $B(0)/\Z_2$ is a cone with vertex $p$.
Let $\rho: \Z_2 \to \End(\R)$ be the representation of $\Z_2$ taking $i$ to $-\Id$. 
Then $\pi=(B(0) \times \R)/\Z_2 \to B(0)/\Z_2$ is an orbibundle. 
Note that $(\pi)^{-1}(x) =\R$, $x\neq p$, and $(\pi)^{-1}(p)=\R_+\cup\{0\}$, so the fibers of $\pi$ are not vector spaces.

\end{exmp}

\begin{exmp}
In Example  \ref{ex1} take the trivial representation $\rho$. Then $(B(0)\times \R)/\Z_2 = B(0)/\Z_2 \times \R$.

\end{exmp}

\subsection{Formal definitions} For the reader's convenience we supply some formal definitions. The full treatment is contained in \cite{Dry} or \cite{Buc}.
\begin{definition}\label{ObDefinition}
\begin{enumerate}
\item An orbifold chart on a topological space $X$ consists of a connected contractible open subset $\widetilde{U}$ of $\R^n$, a finite group $G_U$ 
acting on $\widetilde{U}$ by diffeomorphisms, and a mapping
$\pi_U$ from $\widetilde{U}$ onto an open subset $U$ of $X$ inducing a homeomorphism from the orbit space $G_U \bs \widetilde{U}$ onto $U$.
 We will always assume that the group $G_U$ acts effectively on $\widetilde{U}$.

\item Let $\Ob$ be an orbifold. A point $x$ of $\Ob$ is said to be singular if for some (hence every) orbifold chart $(\widetilde{U}, G_U, \pi_U)$ around $x$, the points in the inverse
image of $x$ in $\widetilde{U}$ have non-trivial isotropy in $G_U$. The isomorphism class of the isotropy group, called the abstract isotropy type of $x$, is independent both of
the choice of point in the inverse image of $x$ in $\widetilde{U}$ and of the choice of chart $(\widetilde{U}, G_U, \pi_U)$ around $x$. Points that are not singular are called regular. 

\end{enumerate}
\end{definition}

\begin{definition}
A Riemannian structure on an orbifold $\Ob$ is an assigment to each 
orbifold chart $(\widetilde{U}, G_U, \pi_U)$ of a $G_U$-invariant Riemannian metric $g_{\widetilde{U}}$ on
$\widetilde{U}$ satisfying a compatibility condition. Every orbifold admits a Riemannian structure.
\end{definition}

\begin{definition}
Let $M$ and $N$ be two orbifolds. A continuous map $f: M \to N$ is smooth if for any $x \in M$ one can find orbifold charts $\mathcal{R} = (\widetilde{U},\Gamma,U,\pi)$ around $x$ 
and  $\mathcal{R}' = (\widetilde{U'},\Gamma',U',\pi_{U'})$ around $f(x)$ and a smooth map $\tilde{f}:\widetilde{U}\to \widetilde{U}'$ such that the following diagram is commutative:
\begin{displaymath}
    \xymatrix{
        \widetilde{U} \ar[r]^{\tilde{f}} \ar[d]_{\pi_U} & \widetilde{U}' \ar[d]^{\pi_{U'}} \\
        U \ar[r]_{f}       & U' }
\end{displaymath}
\end{definition}

For the formal definition of an orbibundle we refer to \cite{Buc}.
\begin{definition}
Let $E \to \Ob$ be an orbibundle. 
A map $f: \Ob \to E$ is a smooth section if $p \circ f = \Id_\Ob$ and if $f$ is smooth as a map between the orbifolds $\Ob$ and $E$.
\end{definition}
\begin{exmp}
Let a group $\Gamma$ act properly discontinously on a manifold $M$. Any representation $\rho: \Gamma \to \End(\R^k)$
 defines an associated orbibundle $M \times \R^k/\Gamma \to M / \Gamma$.
 Then we can identify smooth sections of $M \times \R^k / \Gamma \to M / \Gamma$
with $f \in C^\infty(M, M\times \R^k)$ such that $f(m g) =  \rho(g^{-1}) \cdot f(m)$ for all $g \in \Gamma, m\in M$.
\end{exmp}

\subsection{Pseudodifferential operators on orbibundles}
\label{sectionPDE}
 In this subsection we  explain why the necessary elements of the classical analysis of 
pseudodifferential operators can be applied to orbifolds.

\subsection*{Sobolev spaces}

To define Sobolev norms on an orbifold $\Ob$, first define Sobolev norms locally. Let $\widetilde{U}$ and $\Gamma$ be as in Definition \ref{ObDefinition}. Note that if $\Gamma$ is finite, then $C_0^\infty(\widetilde{U}/\Gamma, (\widetilde{U}\times \R^k)/\Gamma) \cong C_0^\infty(\widetilde{U}, \widetilde{U} \times \R^k)^\Gamma$. 
 Note that $C_0^\infty(\widetilde{U}, \widetilde{U} \times \R^k)$  is equipped with usual Sobolev norm $||\cdot ||_s$, 
and this norm restricts to $\Gamma$-invariant sections. For $f \in C_0^\infty(\widetilde{U}, \widetilde{U} \times \R^k)^\Gamma$,
define $|| f ||_s^\Gamma = \frac{1}{|\Gamma|} || f||_s$.

Next we use an orbifold atlas and a partition of unity to define the Sobolev norm on the space of smooth sections of orbibundle $E \to \Ob$. 
Sobolev norms defined using equivalent atlases will be themselves equivalent. The space 
$H^s(\Ob,E)$ denotes the completion of $C^\infty(\Ob,E)$
with respect to any of these norms.

\begin{remark}
The isomorphism  
$C_0^\infty(\widetilde{U}/\Gamma, (\widetilde{U}\times \R^k)/\Gamma) \cong C_0^\infty(\widetilde{U}, \widetilde{U} \times \R^k)^\Gamma$ 
does not necessarily hold if $\Gamma$ is infinite. For example, let $\gamma$ act on $\R$ by $x \cdot \gamma = x+1$ and put $\Gamma = \cup_{n \in \Z} \gamma^n$. Then $C_0^\infty(\R)^\Gamma = \emptyset$, but $C_0^\infty(\R/\Gamma)\neq \emptyset$.
\end{remark}

\begin{remark}
There is another possibility to introduce Sobolev spaces on manifolds with 
conical singularities; see \cite[p. 21]{Lesch}.
\end{remark}

\subsection*{Pseudodifferential operators}
We recall some basic facts about pseudodifferential operators on orbibundles. 
For more details see \cite[p. 28]{Buc}. 

\begin{definition}
Let $E \to \Ob$ be an orbibundle. A linear mapping $P: C^\infty(\Ob, E) \to C^\infty(\Ob, E)$ is a pseudodifferential 
operator on $E \to \Ob$ of order $m$ if:
\begin{enumerate}
\item
the Schwartz kernel of $P$ is smooth outside any neighborhood of the diagonal in $\Ob \times \Ob$,
\item
for any $x \in \Ob$ and for any orbifold chart $( \tilde{U}, G_U, \phi_U)$ with $x \in U$, the operator 
$C_c^\infty(U, E) \ni f \mapsto P(f) |_U \in C^\infty(U, E) $ is given by the restriction to $G_U$-invariant 
sections of the vector bundle $\tilde{E} \to \tilde{U}$ of an order $m$ pseudodifferential operator on $\tilde{U}$ 
that commutes with the $G_U$-action.
\end{enumerate}
\end{definition}

The Sobolev embedding and the Kondrachov-Rellich theorem are valid as in the case of manifolds in  \cite[p. 60]{Shu}. 
Instead of the original proofs, one chooses a partition of unity 
and reduces the theorems to their local versions in a single chart. As sections over orbifold charts are
$\Gamma$-invariant sections over the corresponding smooth charts, the desired proofs are obtained
by repeating the local arguments from \cite{Shu} verbatim for the subspaces of $\Gamma$-invariant sections.

\begin{remark}
The Sobolev embedding and the Kondrachov-Rellich theorem were proved in \cite{Far} for orbifolds without using pseudodifferential operators.
\end{remark}

\begin{remark}\label{cr}
It follows from the Kondrachov-Rellich theorem that any pseudodifferential operator~$A$ or order $a<0$ is a compact operator.
\end{remark}

Denote by $|| A_\lambda||_{s,s-l}$ the norm of $A$ seen as an operator from $H^s(\Ob, E)$ into $H^{s-l}(\Ob,E)$. 
We want to study the dependence of $||A_\lambda||_{s,s-l}$ on $\lambda$ for large $|\lambda|$. The following theorems are established analogously to the proofs for the smooth case to which we provide references.

\begin{Th} \cite[Theorem 9.1, Ch. II]{Shu}\label{ShuEst}.
Let $A_\lambda \in L^m_d(\Ob, E;\lambda)$ with $l \geqslant m$ and $s \in \R$. Then
\begin{equation*}
\begin{gathered}
|| A_\lambda||_{s,s-l} \leqslant C_{s,l} (1 + |\lambda|^{1/d})^m, \quad \text{if } l \geqslant 0, \\
|| A_\lambda||_{s,s-l} \leqslant C_{s,l} (1 + |\lambda|^{1/d})^{-(l-m)}, \quad  \text{if } l \leqslant 0.
\end{gathered}
\end{equation*}

\end{Th}

\begin{Th} \cite[Theorem 9.3 and Theorem 8.4]{Shu}. \label{spectrBehaviour}
Let $H$ be a generalized Laplacian acting on sections of an orbibundle over a compact orbifold. For a subset $I \subset [0, 2\pi ]$ let 
$$ \Lambda_I := \{ r e^{i \phi}: 0 \leqslant r < \infty, \, \phi \in I\}$$
be the solid angle attached to $I$ and $B_R(0) = \{ x\in \C, \,|x|\le R\}$. Then for every $0 < \varepsilon < \pi / 2$ there exists $R > 0$ such that the 
spectrum of $H$ is contained in the set $B_R(0) \cup \Lambda_{[- \varepsilon, \varepsilon]}$. Moreover the spectrum of $H$ is discrete, and there exists $R \in \R$ such that for $\lambda > R$, 
$$||(H - \lambda)^{-1}|| \leqslant C_{s,l} / |\lambda|.$$
\end{Th}

\section{Functional analysis}\label{sec:!fa}

In this section we refine the necessary functional-analytic facts from \cite[Section 2]{Mu} for the case of compact
orbifolds. Note that we currently do not need our orbifold $\Ob$ to be a global quotient orbifold, but it is necessary that $\Ob$ is compact. Let $E \to \Ob$ be a orbibundle.
Consider the class of elliptic operators
$$ H: C^\infty(\Ob, E) \to C^\infty(\Ob, E),$$
which are perturbations of the Laplace operator $\Delta_E$ by a first order differential operator, i.e.
$$ H = \Delta_E + D,$$
where $D: C^\infty(\Ob, E) \to C^\infty(\Ob, E)$ is a first order differential operator.

It follows from Theorem \ref{spectrBehaviour} that for every $0 < \epsilon < \pi/2$ there exists $R >0$ such that the spectrum of $H$ is contained in 
$B_R(0) \cup \Lambda_{[-\epsilon, +\epsilon]}$. Though
$H$ is not self-adjoint in general, it has nice spectral properties. From the Sobolev space theory for compact orbifolds 
(see Section \ref{sectionPDE}) it follows that $D$ is compact relative to $\Delta_E$ and hence its root vectors are complete. This means that $L^2(\Ob, E)$ is the closure of the algebraic 
direct sum of finite-dimensional $H$-invariant subspaces $V_k$
\begin{equation}
\label{spectralDecomposition}
L^2(\Ob, E) = \overline{  \bigoplus_{k \geqslant 1} V_k},
\end{equation}
such that the restriction of $H$ to $V_k$ has a unique eigenvalue $\lambda_k$, and for each $k$ there exists $N_k \in \mathbb{N}$ such 
that $(H - \lambda_k I)^{N_k} V_k = 0$, and $|\lambda_k| \to \infty$.
For the proof of the Weyl law see Section \ref{seq2}.

Denote by $\spec(H)$ the spectrum of $H$. Suppose that $0 \not \in \spec(H)$. It follows from Theorem \ref{spectrBehaviour} that there exists an Agmon 
angle $\theta$ for $H$ and we can define the square root $H^{1/2}_\theta$.
If $\theta$ is fixed, we simply denote $H_\theta^{1/2}$ by $H^{1/2}$. Note that $H^{1/2}$ is a classical pseudodifferential operator with 
 principal symbol
$$\sigma(H^{1/2}) (x, \sigma) = || \sigma||_x \cdot \Id_E.$$
The principal symbols of $H^{1/2}$ and $\Delta_E^{1/2}$ coincide, hence 
$$ H^{1/2} = \Delta^{1/2}_E+B.$$
Here $B$ is a pseudodifferential operator of order zero, thus it extends to a bounded operator in $L^2(\Ob, E)$. The formula implies the following spectral properties of $H^{1/2}$.
\begin{Lemma}
\label{lemma123456}
The resolvent of $H^{1/2}$ is compact, and the spectrum of $H^{1/2}$ is discrete. There exists $b > 0$ and $c \in \R$ such that the spectrum 
of $H^{1/2}$ is contained in the domain 
$$\Omega_{b,c} = \{ \lambda \in \C: \RRe(\lambda) > c,\, |\IIm(\lambda)|<b \}.$$
\end{Lemma}

It follows from the spectral decomposition (\ref{spectralDecomposition}) that $H^{1/2}$ has the same spectral decomposition with 
eigenvalues $\lambda^{1/2}$, $\lambda \in \spec(H)$ and multiplicities $m(\lambda^{1/2}) = m(\lambda)$. We need to introduce some class of function for further use. 
\begin{definition}
Denote by $\mathcal{P}(\C)$ be the space of Paley-Wiener functions on $\C$, that is 
\begin{equation*}
 \mathcal{P}(\C) = \cup_{R>0} \mathcal{P}^R(\C)
\end{equation*}
with the inductive limit topology, where $\mathcal{P}^R(\C)$ is the space of entire functions $\phi$ on $\C$ such that for every $N \in \mathbb{N}$
there exists $C_N > 0$ such that 
\begin{equation}
\label{PAYLEYWIENER}
|\phi(\lambda)| \leqslant C_N (1+|\lambda|)^{-N} e^{R|\IIm(\lambda)|}, \quad \lambda\in \C.
\end{equation}

\end{definition}
Given $h\in C_0^\infty((-R,R))$, let
$$ \varphi(\lambda) = \frac{1}{\sqrt{2 \pi}} \int_\R h(r) e^{-ir\lambda} dr, \quad \lambda \in \C,$$
be the Fourier-Laplace transform of $h$. Then $\varphi$ satisfies (\ref{PAYLEYWIENER}) for every $N \in \mathbb{N}$, that is $\varphi \in \mathcal{P}^R(\C)$.
Conversely, by the Paley-Wiener theorem, every $\phi \in \mathcal{P}^R(\C)$ is the Fourier-Laplace transform of a function in $C_c^\infty((-R,R))$.

Recall that we are assuming  $0 \not \in \spec(H)$. For $b > 0$ and $d \in \R$ let $\Gamma_{b,d}$ be the contour which is the union of the two half-lines
$L_{\pm b,d} = \{ z \in \C: \IIm(z)=\pm b, \, \RRe(z)\geqslant d \}$ and the semicircle 
$S = \{ +be^{i \theta}: \pi/2 \leqslant \theta \leqslant 3\pi/2 \}$,
oriented clockwise. By Lemma \ref{lemma123456} there exists $b>0$, $d \in \R$ such that $\spec(H^{1/2})$ is contained in the interior of $\Gamma_{b,d}$. 
For an even Paley-Wiener function $\varphi$ put
$$ \varphi(P^{1/2}) = \frac{i}{2\pi}\int_{\Gamma_{b,d}} \varphi(\lambda) (H^{1/2}-\lambda)^{-1} d\lambda.$$
For the definition of $\varphi(H^{1/2})$ in the case when $0 \in \spec(H)$ see \cite[p.~11]{Mu}. 
\begin{Lemma}\cite[Lemma 2.4]{Mu}
$\varphi(P^{1/2})$ is an integral operator with a smoothing kernel.
\end{Lemma}
We need to establish an auxiliary result about smoothing operators. Let
$$ A: L^2(\Ob, E) \to L^2(\Ob, E)$$
be an integral operator with a smooth kernel $K$.

\begin{Lemma}\label{proposition_about_smoothing_operators}
$A$ is a trace class operator and
$$\Tr(A) = \int_\Ob \tr K(x,x) d\mu(x).$$
\end{Lemma}
\begin{proof}
Let $\nabla^E$ be a Hermitian connection in $E$ and let $\Delta_E = (\nabla^E)^*\nabla^E$ be the associated Bochner-Laplace operator. 
Then $\Delta_E$ is a second order
non-negative elliptic operator, which is essentially self-adjoint \cite[p.37, Theorem 3.5]{Buc}. Its spectrum is discrete by
 Theorem \ref{spectrBehaviour}.
The 
rest of the proof follows \cite[Proposition~2.5]{Mu}.
\end{proof}

Now we apply this result to $\varphi(H^{1/2})$, where $\varphi \in \mathcal{P}(\C)$. Let $K_\varphi(x,y)$ be the kernel of $\varphi(H^{1/2})$. Then by Lemma 
\ref{proposition_about_smoothing_operators}, $\varphi(H^{1/2})$ is a trace class operator, and we have
\begin{equation}\label{someTraceEquation}
\Tr \varphi(H^{1/2}) = \int_\Ob \tr K_\varphi(x,x) d \mu(x).
\end{equation}
By Lidskii's theorem \cite[Theorem 8.4]{GK}, the trace is equal to the sum of the eigenvalues of $\varphi(P^{1/2})$, counted 
with their algebraic multiplicities.
One can show that $\varphi(H^{1/2})$ leaves the decomposition (\ref{spectralDecomposition}) invariant and that 
$\varphi(H^{1/2})|_{V_k}$ has unique eigenvalue $\varphi(\lambda_k^{1/2})$. Now, applying Lidskii's theorem and 
(\ref{someTraceEquation}), we get the following

\begin{Lemma}\label{quantifiableconnection}
Let $\varphi\in \mathcal{P}(\C)$ be even. Then we have
\begin{equation}
\sum_{\lambda\in \spec(H)} m(\lambda)\varphi(\lambda^{1/2}) = \int_\Ob \tr K_\varphi (x,x) dx,
\end{equation}
where $m(\lambda)$ is the multiplicity of $\lambda$.
\end{Lemma}


\section{The heat kernel}
\label{seq2}

In this section we consider an orbibundle $E$ over a general compact orbifold $\Ob$ and a generalized Laplacian $H$ acting on sections of $E$. 
Our goal is to construct and study the heat kernel for $H$.



\subsection{Existence and uniqueness of the heat kernel.}
\begin{definition}
We say that $K \in \Gamma((0,\infty) \times \Ob \times \Ob, E \boxtimes E^*)$ is a heat kernel, if it satisfies:

\begin{enumerate}
\item K is  $C^0$ in all three variables, $C^1$ in the first, and $C^2$ in the second,
\item $(\dt + H_x) K(t,x,y) = 0$, where $H_x$ acts on the second variable,
\item $\lim_{t \to +0} K(t,x, \cdot) = \delta_x$ for all $x \in \Ob$.
\end{enumerate}
\end{definition}

The existence of the heat kernel for global quotient orbifolds and the Laplace-Beltrami operator were originally proven in \cite{Donnelly}. In \cite{Dry} it was extended 
 to the case of general compact orbifolds. We generalize this result to a Laplacian type operator $H$ acting on an orbibundle $E$ over a compact orbifold  $\Ob$. 
 The main idea is to take an approximate solution of the heat equation from \cite[Theorem 2.26]{Berl} and use the construction of a parametrix for the heat operator as in \cite{Dry}.

 Cover $\Ob$ with finitely many charts $(\widetilde{U}_\al, G_\al, \pi_\al)$ where $\widetilde{U}_\al / G_\al = U_\al$; the notations are as in Section~\ref{sec:orb}. 
For an open subset $U \subset \widetilde{U_\al}$ denote 
$U_\epsilon := \{(x,y) \in \widetilde{U}_\al \times \widetilde{U}_\al \ | \ d(x,y) < \epsilon\}$, where $d(x,y)$ is the distance between $x$ and $y$. We want to find an approximate heat kernel on $U_\epsilon$.

\begin{Lemma}\cite[Theorem 2.26]{Berl}
\label{localInvariants}
For $l \in \Z$ there exist $u_i(x,y) \in C^\infty(U_\epsilon \times U_\epsilon,  E \boxtimes E^*)$ with $i = 0, \ldots, l,$ such that
\begin{equation}\label{formalSolutionOfHeatEquation}
\left(\dt + H\right) \left[   (4 \pi t)^{-n/2}\cdot \exp^{-d(x,y)^2/4t} \cdot \sum_{i=0}^l u_i(x,y)t^i \right]= (4 \pi)^{-n/2}t^{l-n/2}\cdot \exp^{-d(x,y)^2/4t} \cdot H u_l.
\end{equation}
\end{Lemma}

\begin{remark}
In \cite{Berl} the functions $u_i$ were constructed for 
$\hat{H}: \Gamma(M, E \otimes |\Lambda|^{1/2}) \to \Gamma(M, E \otimes |\Lambda|^{1/2})$ instead of $H: \Gamma(M, E) \to \Gamma(M,E)$, 
where $|\Lambda|$ is density and $M$ is a manifold.
\end{remark}

For the rest of this subsection we follow \cite{Dry}. First we construct a parametrix for the heat operator. Recall that

\begin{definition}
$F\in \Gamma((0, \infty) \times \Ob \times \Ob,  E \boxtimes E^*)$ is a parametrix for the heat operator if:
\begin{enumerate}
\item $F$ is $C^\infty$-smooth,
\item $(\dt + H_x) F(t,x,y)$ extends to a $C^0$ function in all three variables,
\item $\lim_{t \to 0} F(t,x, \cdot )=\delta_x$ for all $x \in \Ob$.
\end{enumerate}
\end{definition}

Identify $\widetilde{U}_\al$ with the unit ball, let $p_\al$ be the center of $U_\al$ and $\widetilde{p}_\al$ the center of $\widetilde{U}_\al$. Let $W_\al$, respectively $V_\al$, 
be the geodesic ball of radius $\epsilon/4$, respectively $\epsilon/2$, centered at $p_\al$, and let $\widetilde{W}_\al$ and 
$\widetilde{V}_\al$ be the corresponding balls centered at  $\widetilde{p}_\al$ in $\widetilde{U}_\al$. We may assume that the family of balls $\{ W_\al\}$ still covers $\Ob$.
For each $\al$ and each non-negative integer $m$ we define 
$\widetilde{H_\al}^{(m)}: \R_+ \times \widetilde{U}_\al \times \widetilde{U}_\al \to \End(\R^k)$ by 
$$ \widetilde{H_\al}^{(m)} (t, \tilde{x}, \tilde{y}) = (4 \pi t)^{-n/2} e^{-d(\tilde{x}, \tilde{y})/4t} (u_0(\tilde{x}, \tilde{y}) + \ldots + t^m u_m(\tilde{x}, \tilde{y})),$$
where the $u_i$ are the invariants from Lemma \ref{localInvariants}. The sum $\sum_{\gamma \in G_\al} \widetilde{H_\al}^{(m)} (t, \tilde{x}, \tilde{y})$ is $G_\al$-invariant in both $\tilde{x}$ and $\tilde{y}$ and thus descends to a function $H_\al^{(m)}$ on $\R_+ \times U_\al \times U_\al$.
Let $\psi_\al: \Ob \to \R$ be a $C^\infty$ cut-off function, which is identically one on $V_\al$ 
and is supported in $U_\al$. Let $\{ \eta_\al \}$ be a partition of unity on $\Ob$ with 
$\supp (\eta_\al) \subset \overline{W_\al}$. 

\begin{definition}
Define $H^{(m)} \in \Gamma(\R_+ \times \Ob \times \Ob, E \boxtimes E^*)$ by
$$H^{(m)}(t,x,y) = \sum_{\al} \psi_\al(x) \eta_\al(y) H^{(m)}_\al (t,x,y).$$
\end{definition}

\begin{Lemma}\cite[page 13]{Dry}.
$H^{(m)}$ is a parametrix for the heat operator on $\Ob$ if $m>n/2$.
\end{Lemma}
From this point, the construction of the heat kernel from the parametrix $H^{(m)}$ is carried out as in \cite[page 210]{Ber}. 
The uniqueness of the heat kernel follows from \cite[Theorem 3.3]{Donnelly}.

\subsection{Computation of the heat asymptotics.}

\begin{Lemma}\label{asympth}

Let $\Ob$ be a Riemannian orbifold, $H$ a generalized Laplacian. Then
\begin{equation}\label{HOHOHO}
 \int_\Ob \tr K(t,x,x) d\vol_\Ob (x) \sim I_0 + \sum_{N \in S(\Ob)} \frac{I_N}{|Iso(N)|}, \quad t \to 0,
\end{equation}
where $S(\Ob)$ is the set of all singular strata and $|Iso(N)|$ is the order of  isotropy at each $p \in N$. There exist $a_k, a_k^N \in \R$ such that
\begin{equation*}
\begin{gathered}
 I_0 \sim t^{-\dim(\Ob)/2} \sum_{k=0}^\infty a_k t^k, \quad t\to 0,\\
 I_N \sim t^{-\dim(N)/2} \sum_{k=0}^\infty a^N_k t^k, \quad t \to 0.
\end{gathered}
\end{equation*}
\end{Lemma}

\begin{proof}[Proof of Lemma \ref{asympth} for global quotient orbifolds]

Let $E' \to M$ be a vector bundle over a compact manifold $M$ with a fiber $V$. Let $\Gamma$ be a finite group of isometries of $M$ and $\rho: \Gamma \to \End(V)$ its representation. Consider an orbibundle $E \to \Ob$, where $E = E' / \Gamma$ and $\Ob = M / \Gamma$. Let $K$ be the heat kernel of $H$
and let $\pi: M \to \Ob$ be the projection.
$$ K^\Ob (t,x,y) = \sum_{\gamma \in \Gamma} \rho^{-1}(\gamma) \cdot K(t, \tilde{x},\gamma \tilde{y}),$$
where $\tilde{x}$ and $\tilde{y}$ are  elements of $\pi^{-1}(x)$ and $\pi^{-1}(y)$, respectively. Then

\begin{equation}\label{OrbiKernel}
\int_\Ob K^\Ob(t,x,x) d \vol_\Ob (x) = \frac{1}{|\Gamma|} \int_M K(t, \tilde{x}, \tilde{x})  d \vol_M (x) + \frac{1}{|\Gamma|}   
\sum_{e \neq \gamma \in \Gamma} \int_M \rho^{-1}(\gamma)  K(t, \tilde{x}, \gamma(\tilde{x})) d \vol_M (x).
\end{equation}
 
We study the asymptotic behavior of (\ref{OrbiKernel}) following \cite{Gil} and modifying the proofs according to the functional analysis from Section \ref{sec:!fa}. The first summand in the right rand side of  (\ref{OrbiKernel}) is treated analogously to \cite[Theorem 1.7.6]{Gil}:
\begin{Th}\label{riiigel}
There exist invariants $a_n(x)$ such that 
 $$\int_M K(t, x, x) dx \sim   \sum_{k=0}^\infty t^{(k-\dim(M))/2} \int_M a^k(x) d\vol_M (x) , \quad t\to 0.$$
\end{Th}
\begin{remark}\label{hase}
The leading coefficient is given by $a_0(x) = (4 \pi)^{-\dim(M)/2}$.
\end{remark}
The second summand of the right hand side of (\ref{OrbiKernel}) is treated analogously to \cite[Lemma 1.8.2]{Gil}:
\begin{Th}\label{drooogel}
Let $\bigcup_i N_i \subset M$ be the fixed point set of $\gamma$ with $m_i = \dim N_i$. There exist invariants $a_n(x)$ which depend functorially on $\gamma$ and a finite number of jets of the symbol of $H$. The invariants $a^{N_i}_i(x)$, $x \in N_i$ and

$$\int_M \rho(\gamma)^{-1} \cdot K(t, x, \gamma x ) d \vol_M(x) \sim \sum_i \sum_{n=0}^\infty t^{(n-m_i)/2} \int_{N_i} a^{N_i}_n(x) d\vol_i (x),$$
where $d\vol_i(x)$ denotes the Riemannian measure on $N_i$.
\end{Th}
\begin{remark}
The invariants $a_n(x)$ and $a^{N_i}_n(x)$ vanish for $n$ odd. A similar result can be found in  \cite[p.438-440]{LR}.
\end{remark}
Applying Theorems \ref{riiigel} and \ref{drooogel} finishes the proof of Lemma \ref{asympth}.
\end{proof}
The proof of Lemma \ref{asympth} for orbifolds that are not global quotient follows  \cite{Dry} with minor modifications using Theorems \ref{riiigel} and \ref{drooogel}, and we omit it. Now we establish the Weyl law. For this we prove:

\begin{Lemma}
Denote by $\lambda_i$, $i=1,2,\ldots$ the eigenvalues of $H$. Then
\begin{equation}\label{HIHIHI}
\sum_i e^{-t \lambda_i} \sim I_0 + \sum_{N \in \Ob} \frac{I_N}{|Iso(N)|},
\end{equation}
where $I_0$, $I_N$ are from Lemma \ref{asympth}.
\end{Lemma}

\begin{proof}
The operator $e^{-t H}$ is Hilbert-Shmidt, because it has the square-integrable kernel $K$ and $\Ob$ is compact. 
Hence, $e^{-t H}$ is compact, and its spectrum consists of countably many eigenvalues eigenvalues $\al_i$, $i\in\N$.
As $e^{-tH} = e^{-tH/2}\cdot e^{-tH/2}$, $e^{-tH}$ is also of trace class. Hence, by Lidskii's theorem
\begin{equation}\label{eq1212}
\sum_i \al_i = \int_\Ob \Tr K(t,x,x) \, d \vol_\Ob(x).
\end{equation}
The eigenvalues $\al_i$ of $e^{-tH}$ and their algebraic multiplicities can be determined as in \cite[p.~13]{Mu}:  
\begin{equation}\label{eq1313}
\al_i = e^{-t \lambda_i} 
\end{equation}
Substituting (\ref{eq1313}) and (\ref{eq1212}) into (\ref{HOHOHO}), we obtain (\ref{HIHIHI}).
\end{proof}
Using Tauberian theorem \cite[Chapt. II, §14]{Shu} and Remark \ref{hase} we obtain:
\begin{Th}[The Weyl law]\label{darthveider}
Let $N(r) = \#\{\lambda \in \spec(H), |\lambda|\le r\}$ be the counting function for the spectrum of~$H$. Then
$$N(r) = \frac{\rk (E) \vol(\Ob)}{(4 \pi)^{n/2} \Gamma(n/2+1)} r^{n/2} + o(r^{n/2}), \quad r \to \infty.$$
\end{Th}

\section{The Selberg trace formula}\label{nyash}

In this section we give a description of the kernel $K_\varphi$ of the smoothing operator $\varphi(P^{1/2})$ in terms of the solution of the wave equation. For technical reasons we impose some restrictions on the orbifold~$\Ob$, namely assume $\Ob = \Gamma \bs M$, where $\Gamma$ is as in the following Lemma:
\begin{Lemma}[Selberg Lemma]\label{seelem}
A finitely generated group $\Gamma$ of matrices over a field of characteristic zero has a normal torsion free subgroup $\Gamma_0$ of finite index.
\end{Lemma}
It follows immediately that $\Gamma_0 \bs M$ is a manifold. Assume additionally that $M$ is a universal covering for $\Gamma_0 \bs M$. Consider the wave equation on the orbibundle $E \to \Ob$ associated to a representation~$\rho$ of~$\Gamma$:
\begin{equation}\label{waveEquation}
\partial^2 (u)/\partial t^2 + H u = 0, \quad u(0,x) = f(x), \quad u_t(0, x) = 0
\end{equation}
for $u(t;f) \in C^\infty (\R \times \Ob, E)$.
\begin{Lemma}\cite[Proposition 3.1]{Mu} \label{prop1}
For each $f \in C^\infty(\Ob, E)$ there is a unique solution $u(t;f) \in C^\infty (\R \times \Ob, E)$ of the wave equation (\ref{waveEquation}). 
Moreover for every $T>0$ and $s \in \R$ there exists $C > 0$ such that for every $f \in C^\infty(\Ob, E)$
\begin{equation}\label{Sobeq}
|| u(t;\cdot ) ||_s \leqslant C || f||_s, \quad |t|\leqslant T
\end{equation}
\end{Lemma}
\begin{proof}
Note that $\Gamma_0 \backslash M$ is a manifold by Lemma \ref{seelem}, hence \cite[Proposition 3.1]{Mu} is valid. Moreover, for every $f \in C^\infty(\Ob, E)$ denote by $g$ its pull-back to $C^\infty(\Gamma_0 \bs M, E)$. Then 
\begin{equation*}
\begin{gathered}
||g||_{\Gamma_0 \bs M} = \sum_{\gamma \in [\Gamma_0 : \Gamma]} ||\rho(\gamma) \cdot f||_{\Gamma \bs M},\quad \textnormal{hence}\\
[\Gamma_0 : \Gamma] \cdot \min_{\gamma \in [\Gamma_0 : \Gamma]} |\rho(\gamma)| \cdot || f||_{\Gamma \bs M} \le ||g||_{\Gamma_0 \bs M} \le [\Gamma_0 : \Gamma] \cdot \max_{\gamma \in [\Gamma_0 : \Gamma]} |\rho(\gamma)| \cdot || f||_{\Gamma \bs M},
\end{gathered}
\end{equation*}
hence (\ref{Sobeq}) follows.
\end{proof}

\begin{Lemma}\cite[Proposition 3.2]{Mu}\label{Prop3.2}
Let $\varphi \in \mathcal{P}(\C)$ and 
 $\widehat{\phi}$ be the Fourier transform of $\varphi|_\R$. Then for every $f \in C^\infty(\Ob, E)$ we have
\begin{equation}
\varphi(H^{1/2}) f = \frac{1}{\sqrt{2\pi}} \int_\R \widehat{\phi}(t) u(t;f) dt.
\end{equation}
\end{Lemma}
Let $d(x,y)$ denote the geodesic distance of $x,y \in M$. For $\delta > 0$ define $U_\delta := \{ (x,y)\in M \times M: d(x,y) < \delta\}$.

\begin{Lemma} \cite[Proposition 3.3]{Mu} \label{Prop3.3}
There exists $\delta > 0$ and
 $H_\phi \in C^\infty(M \times M, Hom(\tilde{E},\tilde{E}))$ with 
 $\supp H_\phi \subset U_\delta$,
such that for all $\psi \in C^\infty (M,\tilde{E}) $ we have

\begin{equation*}
\frac{1}{\sqrt{2\pi}} \int_\R \widehat{\phi}(t) w(t,\tilde{x},\psi) dt 
 = \int_{M} H_\varphi(\tilde{x},\tilde{y})
  (\psi(\tilde{y})) d \tilde{y}.
\end{equation*}
\end{Lemma}
From now on denote by $\tilde{f}$ the pull-back of $f \in C^\infty(\Gamma \bs M)$ to $M$. Using Lemmas \ref{Prop3.2} and \ref{Prop3.3} we obtain
\begin{equation}\label{mimimi}
\varphi(H^{1/2})f(\tilde{x}) = \int_{M} H_\varphi(\tilde{x},\tilde{y})(\tilde{f}(\tilde{y}))\, d\tilde{y}
\end{equation}
for all $f\in C^\infty(X,E)$.
Using Proposition (3.3) together with (3.15) and Proposition 3.2, we obtain
\begin{equation}\label{3.16}
\varphi(H^{1/2})f(\tilde{x}) = \int_M H_\varphi(\tilde{x}, \tilde{y})(\tilde{f}(\tilde{y})) d\tilde{y} \quad \forall f \in C^\infty(\Ob, E).
\end{equation}
Let $F \subset M$ be a fundamental domain for the action of $\Gamma$ on $M$. Then we get 
\begin{equation*}
\begin{gathered}
\int_M H_\varphi (\tilde{x}, \tilde{y}) (\tilde{f}(\tilde{y})) d\tilde{y} = \sum_{\gamma \in \Gamma}  \int_{\gamma F} H_\varphi(\tilde{x}, \tilde{y})(\tilde{f} \tilde{y}) d \tilde{y} = \\
\sum_{\gamma \in \Gamma} \int_F H_\varphi(\tilde{x}, \gamma \tilde{y}) (\tilde{f}(\gamma \tilde{y})) d \tilde{y} =\\
\int_F \left(\sum_{\gamma \in \Gamma} H_\varphi(\tilde{x}, \gamma \tilde{y}) \circ \rho(\gamma)\right)\cdot (\tilde{f}(\tilde{y})). d\tilde{y}.
\end{gathered}
\end{equation*}
Combining this expression with (\ref{3.16}), it follows that the kernel $K_\varphi$ of $\varphi(H^{1/2})$ is given by
$$ K_\varphi(x,y) = \sum_{\gamma \in \Gamma} H_\varphi(\tilde{x}, \gamma \tilde{y})\circ \rho(\gamma).$$
Together with Lemma \ref{quantifiableconnection} we obtain
\begin{Lemma}\cite[Proposition 3.4]{Mu}
Let $\varphi \in \mathcal{P}$ be even. Then we have
\begin{equation*}
\sum_{\lambda \in \spec(H)} m(\lambda) \varphi(\lambda^{1/2}) = \sum_{\gamma \in \Gamma} \int_F \tr(H_\varphi(\tilde{x}, \gamma \tilde{x}) \circ \rho(\gamma)) d\tilde{x}.
\end{equation*}
\end{Lemma}

\subsection{The twisted Bochner-Laplace operator}\label{Sec8}
In this section we follow \cite[Section 4]{Mu}  to introduce the twisted non-selfadjoint Laplacian $\Delta_{E, \rho}^\#$.
Let $\Ob$ be a global quotient orbifold $\Ob = \Gamma \bs M$. Let $E \to \Ob$ be a complex orbibundle over an orbifold $\Ob$ with covariant derivative $\nabla$. The connection Laplacian and Bochner-Laplace operators are defined as in the case for manifolds.

Recall the definition of associated vector orbibundles. Suppose that $M$ is simply connected. Let $\rho: \Gamma \to GL(V)$ be a finite-dimensional representation. Define an equivalence relation on 
$$ M \times_\rho V : \quad [h,v] \sim [gh, \rho(g^{-1})v]$$
and let $F = M \times_\rho V$. Then $F \to \Ob$ is an orbibundle over $\Ob$. Pick a flat connection $\nabla^F$ on $F$.
This is equivalent to choosing an arbitrary $G$-invariant flat connection on $\HH \times V \to \HH$.

Let $E$ be a Hermitian vector orbibundle over $\Ob$ with a Hermitian connection $\nabla^E$. We equip $E \otimes F$ with the product 
connection $\nabla^{E \otimes F}$, defined by $$\nabla_Y^{E \otimes F} = \nabla_Y^E \otimes 1 + 1 \otimes \nabla_Y^F$$
for $Y \in C^\infty(M, TM)$. Let $\Delta_{E, \rho}^\#$ be the connection Laplacian associated to  $\nabla^{E \otimes F}$.
Let $\widetilde{E}$ and $\widetilde{F}$ be the pullback to $M$ of $E$ and $F$ respectively. Then $\widetilde{F} \cong M \times V$ and
$$C^\infty(M, \widetilde{E} \times \widetilde{F}) \cong C^\infty(M, \widetilde{E}) \otimes V.$$
It follows that if we take the lift $\widetilde{\Delta}^\#_{E, \rho}$ of $\Delta^\#_{E,\rho}$ and the lift $\widetilde{\Delta}_E$ of $\Delta_E$ to $M$,
\begin{equation}\label{princesssweety}
\widetilde{\Delta}_{E, \rho}^\# = \widetilde{\Delta}_E \otimes \Id.
\end{equation} 
Then the unique solution of the equation 
$$(\partial^2/\partial t^2 + \widetilde{\Delta}_{E, \rho}) u(t;\psi) = 0, \quad u(0; \psi)=\psi, \, u_t(0;\psi)=0 $$
is given by
$$u(t;\psi) = \left( \cos(t \widetilde{\Delta}_E) \otimes Id \right) \psi.$$
Let $\varphi \in \mathcal{P}(\C)$ be even and let $k_\varphi(\tilde{x},\tilde{y})$ be the kernel of
$$ \varphi \left( (\widetilde{\Delta}_E)^{1/2}  \right) = \frac{1}{\sqrt{2\pi}} \int_\R \hat{\varphi}(t) \cos(t (\widetilde{\Delta}_E)^{1/2}) \, dt.$$

Then the kernel $H_\varphi$  is given by $H_\varphi(\tilde{x},\tilde{y}) = k_\varphi(\tilde{x}, \tilde{y}) \otimes Id$.
Let $R_\gamma: \widetilde{E}_{\tilde{y}} \to \widetilde{E}_{\gamma \tilde{y}}$
 be the canonical isomorphism. Then it follows from  that the kernel
of the operator $\varphi(\widetilde{\Delta}_E)^{1/2})$ is given by
$$ K_\varphi(x,y)=\sum_{\gamma \in \Gamma} k_\varphi(\tilde{x},\gamma\tilde{y})\circ(R_\gamma \otimes \rho(\gamma)).$$
\begin{Lemma}\label{PRETRAAACE}
Let $F_\rho$ be a flat vector orbibundle over $\Ob$, associated to a finite-dimensional complex representation $\rho: \Gamma \to GL(V)$. Let $\Delta_{E, \rho}^\#$
be the twisted connection Laplacian acting in $C^\infty(\Ob, E \otimes F_\rho)$. Let $\varphi\in \mathcal{P}(\C)$ be even and denote by $k_\varphi(\tilde{x},\tilde{y})$
the kernel of $\varphi\left((\widetilde{\Delta}_E)^{1/2}  \right)$. Then we have
$$\sum_{\lambda \in spec(\Delta_{E,\rho}^\#)} m(\lambda)\varphi(\lambda^{1/2}) = \sum_{\gamma \in \Gamma} \tr \rho(\gamma) \int_F \tr (k_\varphi(\tilde{x},\gamma\tilde{y})\circ R_\gamma) \, d\tilde{x}.$$
\end{Lemma}
\begin{remark}
Lemma \ref{PRETRAAACE} was proved for manifolds in \cite[Proposition 4.1]{Mu}.
\end{remark}
\subsection{Locally symmetric subspaces and the pre-trace formula}\label{Sec9}

Let $G$ be a connected semisimple real Lie group of non-compact type with finite center. Let $K \subset G$ be a maximal compact subgroup of $G$. Let $G=KAN$ be its Iwasawa decomposition, and let $M$ be the centralizer of $A$ in $G$. Denote by $\mathfrak{g}$, $\mathfrak{k}$ and $\mathfrak{m}$
the Lie algebras of $G$, $K$ and $M$, respectively. Let
$$ \mathfrak{g} = \mathfrak{p} \oplus \mathfrak{k}$$
be the Cartan decomposition. Denote $S := G / K$ a Riemannian symmetric space of non-positive curvature, whose invariant metric is obtained translating the restriction
of the Killing form to $\mathfrak{p} \cong T_e(G / K)$.
Let $\tau: K \to GL(V_\tau)$ be a finite-dimensional unitary representation of $K$, and let
$$\tilde{E}_\tau = (G \times V_\tau)/K \to G/K$$
be the associated homogeneous vector bundle, where $K$ acts on $G \times V_{\tau}$ by
$$ (g,v) k = (gk, \tau(k^{-1}) v), \quad g\in G, \, k\in K, \, v\in V_\tau.$$
Let
\begin{equation}\label{5.2}
C^\infty(G;\tau) := \{ f:G\to V_\tau \, | \, f\in C^\infty, f(gk)=\tau(k^{-1})f(g), \quad g\in G, k\in K    \}.
\end{equation}
Similarly, we denote by $C_c^\infty(G; \tau)$  the subspace of compactly supported functions in $C^\infty(G;\tau)$ and by $L^2(G;\tau)$ 
the completion of $C_c^\infty(G;\tau)$ with respect to the inner product 
$$ \langle f_1,f_2\rangle = \int_{G / K} \langle f_1(g), f_2(g)\rangle d\dot{g}.$$
There is a canonical isomorphism \cite[page 4]{Mi}
\begin{equation}\label{5.3}
C^\infty(S, \tilde{E}_\tau) \cong C^\infty(G; \tau).
\end{equation}
Similarly, there are isomorphisms $C_c^\infty(S, \tilde{E}_\tau) \cong C_c^\infty(G; \tau)$ and $L^2(S, \tilde{E}_\tau) \cong L^2(G;\tau)$.
Let $\nabla^\tau$ be the canonical $G$-invariant connection on $\tilde{E}_\tau$ defined by
$$ \nabla_{g_* Y}^\tau f(g K) = \frac{d}{dt}\Big| _{t = 0} (g \exp(tY))^{-1} f(g \exp(tY) K),$$
where $f \in C^\infty(G;\tau)$ and $Y \in \mathfrak{p}$. Let $\tilde{\Delta}_\tau$ be the associated Bochner-Laplace operator.
Then $\tilde{\Delta}_\tau$ is $G$-invariant, that is $\tilde{\Delta}_\tau$ commutes with the right action of $G$ on $C^\infty(S, \tilde{E_\tau})$.
Let $\Omega \in Z(\mathfrak{g}_\C)$ and $\Omega_K \in Z(\mathfrak{k}_\C)$ be the Casimir elements of $G$ and $K$, respectively.
Assume that $\tau$ is irreducible. Let $R$ denote the right regular representation of $G$ on $C^\infty(G; \tau)$. Then with respect to (\ref{5.3}), we have
\begin{equation}\label{20}
\tilde{\Delta}_\tau = -R(\Omega) + \lambda_\tau \Id,
\end{equation}
where $\lambda_\tau = \tau(\Omega_K)$ is the Casimir eigenvalue of $\tau$ \cite[Proposition 1.1]{Mi}. We note that $\lambda_\tau \geqslant 0$.

Let $\varphi \in C_0^\infty(\C)$ be an even function. Then $\varphi(\tilde{\Delta}_\tau)$ is a $G$-invariant integral operator, therefore its kernel $k_\varphi$ satisfies
$$ k_\varphi(g \tilde{x}, g\tilde{y}) = k_\varphi(\tilde{x}, \tilde{y}), \quad g \in G, \quad \tilde{x}, \tilde{y} \in \HH^{2n+1}.$$
\begin{remark}
In the scalar case, i.e.~when $\tau$ is a trivial representation, this is a point-pair invariant considered originally by Selberg \cite{Sel}.
\end{remark}
With respect to the isomorphism (\ref{5.3}) $k_\varphi$ can be identified with
a compactly supported smooth function
$$ h_\varphi: G \to \End(V_\tau),$$
which satisfies
$$ h_\varphi(k_1 g k_2) = \tau(k_1) \circ h_\varphi(g) \circ \tau(k_2), \quad k_1, k_2 \in K.$$
Then $\varphi(\tilde{\Delta}_\tau^{1/2})$ acts by convolution:
\begin{equation}
\left( \varphi(\tilde{\Delta}_\tau^{1/2}) f \right) (g_1) = \int_G h_\varphi(g_1^{-1} g_2)(f(g_2)) dg_2.
\end{equation}
Let $E_\tau = \Gamma \backslash\tilde{E}_\tau$ be the locally homogeneous vector orbibundle over $\Gamma \backslash S$ induced by $\tilde{E}_\tau$.
 Let $\chi: \Gamma \to GL(V_\chi)$ be a finite-dimensional complex representation and let $F_\chi$ be the associated flat vector bundle over $\Gamma \backslash S$.
 Let $\Delta_{\tau, \chi}^\#$ be the twisted connection Laplacian acting in $C^\infty(\Gamma \backslash S, E_\tau \otimes F_\xi)$. Then it follows from (4.3)
 that the kernel $K_\varphi$ of $\varphi(\tilde{\Delta}_\tau)^{1/2}$ is given by
\begin{equation}
 K_\varphi(g_1 K, g_2 K) = \sum_{\gamma \in \Gamma} h_\varphi(g_1^{-1} \gamma g_2) \otimes \chi(\gamma).
\end{equation}
By Proposition 4.1 we get
\begin{equation}\label{gebrochenherze}
\sum_{\lambda\in \spec(\Delta_{E, \rho}^\#)} m(\lambda) \varphi(\lambda^{1/2}) = \sum_{\gamma \in \Gamma} \tr \chi(\gamma) \int_{\Gamma \backslash G} \tr h_\varphi(g^{-1} \gamma g) d\dot{g}.
\end{equation}
Now collect the terms in the right hand side of (\ref{gebrochenherze})  according to their conjugacy classes. Given $\gamma \in \Gamma$,
 denote by $\{ \gamma\}_\Gamma$ its $\Gamma$-conjugacy class, by $\Gamma_\gamma$ and $G_\gamma$ 
the centralizers of $\gamma$ in $\Gamma$ and $G$, respectively. Separating $\{ e\}_\Gamma$, we obtain a pre-trace formula.

\begin{Proposition}\label{propSTF} [Pre-trace formula]
For all even $\varphi\in \mathcal{P}(\C)$ we have:
\begin{equation*}
\begin{gathered}
\sum_{\lambda\in \spec(\Delta_{E, \rho}^\#)} m(\lambda) \varphi(\lambda^{1/2}) = \dim(V_\chi) \vol (\Gamma \backslash S) \tr h_\varphi(e) + \\
+ \sum_{\{\gamma\}_\Gamma \neq \{e\}} \tr \chi(\gamma) \vol (\Gamma_\gamma \backslash G_\gamma)  \int_{G_\gamma \backslash G} \tr h_\varphi(g^{-1} \gamma g) d\dot{g}.
\end{gathered}
\end{equation*}
\end{Proposition}

\subsection{Hyperbolic space.}\label{rankonemonkey}


Proposition \ref{propSTF} is valid for $\Gamma = \Gamma \bs G / K$, where $G$ and $K$ are as in Subsection \ref{Sec9}. Here we specialize it to the case of odd-dimensional hyperbolic compact orbifolds. For this take $G = SO_0(1,2n+1)$, $K = SO(2n+1)$, so $S = G/K$, equipped with an invariant metric as in the previous subsection, is isometric to the hyperbolic space $\HH^{2n+1}$. Let $\Gamma \subset G$ be a discrete group acting properly discontinuously on $\HH^{2n+1}$ and suppose $\Gamma \bs \HH^{2n+1}$ is a  compact orbifold. This implies $\Gamma \setminus \{e\}$ consists of two types of elements:
\begin{definition}
$\gamma \in \Gamma$ is called  elliptic if it is of finite order.
\end{definition}
\begin{definition}
$\gamma \in \Gamma$ is called hyperbolic if 
$$ l(\gamma) := \inf_{x \in \HH^{2n+1}} d(x, \gamma x) > 0.$$
\end{definition}

\begin{Lemma}
\cite[Lemma 6.6]{Wal}
For $\gamma$ hyperbolic there exists $g \in G$, $m_\gamma \in SO(2n)$, $a_\gamma \in A^+$ s.t. 
$$ g \gamma g^{-1} = m_\gamma a_\gamma.$$
Here $t_\gamma$ is unique, and $m_\gamma$ is determined up to conjugacy in $SO(2n)$.
\end{Lemma}

By the Lemma \ref{seelem}, there exists a normal torsion free subgroup $\Gamma'$ of finite index in~$\Gamma$. Denote by $\Gamma'_\gamma$ the centralizer of $\gamma$ in $\Gamma'$.  Note that $\Gamma_\gamma$ does not necessarily consist of hyperbolic elements only, so $\Gamma'_\gamma$ and $\Gamma_\gamma$ may be different; to measure the difference introduce 
$$ v(\gamma) := \vol (\Gamma_\gamma \bs G_\gamma) / \vol(\Gamma'_\gamma \bs G_\gamma).$$
 The structure of $\Gamma'_\gamma$ is known:
\begin{Proposition}\label{verstandnurihresprachenicht}
 The centralizer $\Gamma'_\gamma$  is an infinite cyclic group. Moreover there exists a hyperbolic primitive element
$\gamma_0 \in \Gamma'$ such that $ \Gamma'_\gamma = \langle \gamma_0 \rangle$
and $ \gamma = \gamma_0^{n_{\Gamma'} (\gamma)}$.
\end{Proposition}

We are interested in calculating $\int_{G_\gamma \backslash G} \tr h_\varphi(g^{-1} \gamma g) d\dot{g}$ appearing in Proposition \ref{propSTF}  more precisely. Recall that 
for $\sigma \in \widehat{SO(2n)}$ and $\lambda\in \R$ one can define the unitarily induced representations $\pi_{\sigma, \lambda}$ as in \cite[p. 177]{Wal}, let $\Theta_{\sigma,\lambda}$ denote the character of $\pi_{\sigma, \lambda}$.
For hyperbolic $\gamma$  we slightly modify \cite[Theorem 6.7]{Wal}:
\begin{Lemma}\label{moyaaktrisa} Let $\gamma \in \Gamma$ be a hyperbolic element. Then the following holds.
\begin{equation*}
\begin{gathered}
\vol(\Gamma_\gamma \bs G_\gamma)   \tr(\chi(\gamma))   \int_{G_\gamma \backslash G} \tr h_\varphi(g^{-1} \gamma g) d\dot{g} = \\    \frac{\tr(\chi(\gamma)) v(\gamma) l(\gamma_0)}{2 \pi D(\gamma)} \cdot \sum_{\sigma \in \widehat{SO(2n)}} \overline{\tr \sigma(\gamma)} \int_{\R} \Theta_{\sigma, \lambda}(h_\varphi)\cdot e^{-il(\gamma) \lambda} d \lambda,
\end{gathered}
\end{equation*} 
where 
\begin{equation}
D(\gamma) = e^{-n l(\gamma)} \left|\det(\Ad(m_\gamma a_\gamma)|_{\mathfrak{n}} - \Id)\right|.
\end{equation}
\end{Lemma}
\begin{remark}
The difference between \cite[Theorem 6.7]{Wal} and Lemma \ref{moyaaktrisa} is in the presence of~$v(\gamma)$ and $\tr(\chi(\gamma))$. If $\Gamma$ contains no elliptic elements, one has $v(\gamma) = 1$.  
\end{remark}

\subsection{Orbital integrals for elliptic elements.}\label{sec:odddim}
It remains to calculate the orbital integrals
$$E_\gamma(h_\varphi) := \int_{G_\gamma \backslash G} \tr h_\varphi(g^{-1} \gamma g) d\dot{g}$$ for $\gamma \in \Gamma$ elliptic.  We may assume $\gamma$ is of the form: 
$$\gamma = \left( \begin{smallmatrix}
   1\\
   & \ddots  \\ 
   &  & 1\\
   & & & R_{\phi_{k+1}}\\
   & & & & \ddots\\
   & & & & & R_{\phi_{n+1}}
 \end{smallmatrix} \right),$$
where $R_{\phi} = \left(\begin{smallmatrix}
\cos \phi & \sin \phi \\
-\sin \phi & \cos \phi
\end{smallmatrix}\right)$, $\phi \in (0, 2\pi)$. Note that the stabilizer $G_\gamma$ does not equal $SO_0(1,1)\times (\T^{2})^{n}$ if $k>1$ or if  $\phi_i = \phi_j$, hence in general $\gamma$ is not a regular element. For further use we want to approximate it by a sequence of regular elements $\gamma_\varepsilon$ parametrized by~$\varepsilon$:
\begin{equation}\label{varepsilon}
\gamma_\varepsilon = \left(  \begin{smallmatrix}
1\\
& 1\\
& & R_{\varepsilon_2}\\
& && \ddots  \\ 
& && & R_{\varepsilon_k}\\
& & && & R_{\varepsilon_{k+1}}\\
& & && & & \ddots\\
& & & && & & R_{\varepsilon_{n+1}}
\end{smallmatrix} \right),
\end{equation}
where $\varepsilon_i \in \R$ are chosen in the following way: all $\varepsilon_i$ are different and 
\begin{equation*}
\begin{gathered}
\lim_{\varepsilon \to 0} \varepsilon_i = 0, \quad i \le k,\\
\lim_{\varepsilon \to 0} \varepsilon_i = \phi_i, \quad i > k.\\
\end{gathered}
\end{equation*}
The strategy for the subsection is the following: first we recall how to calculate  $E_{\gamma_\varepsilon}(h)$, second we apply a certain element of the symmetric algebra $S(\mathfrak{b}_\C)$, and set $\varepsilon = 0$ to obtain $E_{\gamma}(h)$. 
The following lemma is based on \cite[Theorem 13.1]{Kna} with adaptations made it suitable for our setting.
\begin{Lemma}\label{neptune}
The orbital integral $E_{\gamma_\varepsilon}(h_{\varphi})$ can be expressed as
\begin{equation}\label{orbitint}
E_{\gamma_\varepsilon}(h_\varphi) = C\cdot \sum_{\sigma \in \widehat{SO(2n)}} \int_{\R}   \sum_{s \in W} \det (s)
\left( \xi_{-s(\Lambda(\sigma) + \delta_M)- i e_1 \nu}(\gamma_\varepsilon) \right)\cdot \Theta_{\sigma,\nu} (h_\varphi) d \nu,
\end{equation}
where $C \in \R \setminus\{0\}$ does not depend on $\varepsilon$. The quantities above which have not yet been defined, shall be defined in the course of the proof. The sum in (\ref{orbitint}) is finite, because $h_\varphi$ is $K$-finite.
\end{Lemma}
\begin{proof}
Denote by $E_{i,j}$ the matrix in $\mathfrak{so}(1,2n+1)$ whose $(i,j)$-entry is 1 and the other entries are 0. Let
\begin{equation*}
\begin{gathered}
H_1 := E_{1,2} + E_{2,1},\\
H_j := i (E_{2j-1,2j} - E_{2j,2j-1}), \quad j = 2, \ldots, n+1.
\end{gathered}
\end{equation*}
Then 
\begin{equation}\label{abizzzyana}
\gamma_\varepsilon = \exp(\varepsilon_2 H_2 + \ldots + \varepsilon_{n+1} H_{n+1}).
\end{equation} 
In the notation of Subsection \ref{Sec9} we have  $\mathfrak{m} = \mathfrak{so}(2n)$, $\mathfrak{a} = \mathfrak{so}(1,1)$, 
$$ \mathfrak{a} = \R H_1,$$ 
and
$$ \mathfrak{b} = i \R H_2 + \ldots + i \R H_{n+1}$$
is the standard Cartan subalgebra of $\mathfrak{so}(2n)$. Moreover,
$ \mathfrak{h} := \mathfrak{a} \oplus \mathfrak{b}$
is a Cartan subalgebra of $\mathfrak{so}_0(1,2n+1)$. Define $e_i \in \mathfrak{h}_{\C}^*$ with $i = 1, \ldots, n+1$, by
\begin{equation}\label{makaronina}
 e_i(H_j) = \delta_{i,j}, \, 1\le i,j\le n+1.
\end{equation}
Then the sets of roots of $(\mathfrak{so}_0(1,2n+1)_\C, \mathfrak{h}_{\C})$ and of  $(\mathfrak{so}(2n)_{\C}, \mathfrak{b}_{\C})$ are given by
\begin{equation*}
\begin{gathered}
\Delta(\mathfrak{so}_0(1,2n+1)_\C, \mathfrak{h}_{\C}) = \{ \pm e_i \pm e_j, 1 \le i < j \le n+1\},\\
\Delta(\mathfrak{so}(2n)_{\C}, \mathfrak{b}_{\C}) = \{  \pm e_i \pm e_j, 2\le i < j \le n+1   \}.
\end{gathered}
\end{equation*}
We fix the positive systems of roots by 
\begin{equation*}
\begin{gathered}
\Delta^+(\mathfrak{so}_0(1,2n+1)_\C, \mathfrak{h}_{\C}) = \{  e_i \pm e_j, 1 \le i < j \le n+1\},\\
\Delta^+(\mathfrak{so}(2n)_{\C}, \mathfrak{b}_{\C}) = \{   e_i \pm e_j, 2\le i < j \le n+1   \}.
\end{gathered}
\end{equation*}
Denote by $W$ the Weyl group of $\Delta(\mathfrak{so}(2n)_{\C}, \mathfrak{b}_{\C})$. The half-sum of positive roots $\Delta^+(\mathfrak{so}(2n)_{\C}, \mathfrak{b}_{\C})$ equals 
\begin{equation}\label{halfsummy}
\delta_M = \sum_{j=2}^{n+1} \rho_j e_j,\quad \rho_j = n+1-j.
\end{equation}

Recall that $\widehat{SO(2n)}$ is the set of equivalence classes of irreducible finite-volume representations of $SO(2n)$, and for $\sigma \in \widehat{SO(2n)}$ let $\Lambda(\sigma)$ be its highest weight.
To an irreducible finite-dimensional representation of $SO_0(1,2n+1)$ with the highest weight $\Omega = k_1 e_1  + \ldots +  k_{n+1} e_{n+1}$, $k_i \in \frac{1}{2} \Z$, we can associate a character on a Cartan subgroup $\exp(\mathfrak{h})$ according to \cite[p. 84]{Kna}:
\begin{equation}\label{kukushechka}
\xi_{\Omega}(\exp(H)) = e^{\Omega(H)}.
\end{equation}
Substituting everything in \cite[Theorem 13.1]{Kna} proves Lemma \ref{neptune}.
\end{proof}
\begin{remark}\label{inspirationpublic}
For a representation with the highest weight $\Omega = i \lambda e_1 + k_2 e_2 + \ldots + k_{n+1} e_{n+1}$ with $\lambda \in \R$ and $\xi_\varepsilon$ as in (\ref{varepsilon})
$$\xi_{\Omega}(\gamma_\varepsilon) = e^{k_2 \varepsilon_2 + \ldots + k_{n+1} \varepsilon_{n+1}}$$ 
does not depend on $\lambda$.
\end{remark}
\begin{proof}
Follows from (\ref{abizzzyana}), (\ref{makaronina}) and (\ref{kukushechka}).
\end{proof}
Recall that the Killing form $B(X,Y)$ on $\mathfrak{so}(1,2n+1) \times \mathfrak{so}(1,2n+1)$ 
is defined by $B(X,Y) = \Tr(\ad(X) \circ \ad(Y))$. We define a symmetric bilinear form $\langle \cdot,\cdot\rangle$ on $\mathfrak{so}(1,2n+1)$ by
$$ \langle Y_1, Y_2 \rangle := \frac{1}{4n} B(Y_1, Y_2), \quad Y_1, Y_2 \in \mathfrak{so}(1,2n+1).$$
For $\alpha \in \Delta^+(\mathfrak{so}_0(1,2n+1)_\C, \mathfrak{h}_{\C})$ there exists a unique $H'_\alpha \in \mathfrak{so}_0(1,2n+1)_\C$ such that $B(H,H'_\alpha) = \alpha (H)$ for all $H \in \mathfrak{so}_0(1,2n+1)_\C$. One has $\alpha(H'_\alpha) \neq 0$. Denote
$$ H_\alpha := \frac{2}{\alpha(H'_\alpha)} H'_\alpha.$$
Note that $H_{\pm e_i \pm e_j} = \pm H_i \pm H_j.$ Without loss of generality assume that all $\varphi_i$ are different, then the stabilizer $G_\gamma$ of $\gamma$ is equal to $\T^k \times SO_0(1, 2k-1)$. The root system 
for $G_\gamma$ can be written as
$$ \Delta_\gamma(\mathfrak{g}_{ \C}, \mathfrak{h}_\C) = (\pm e_i \pm e_j, 1 \leqslant i < j \leqslant k ). $$
We can choose an ordering such that 
$$ \Delta^+_\gamma(\mathfrak{g}_{ \C}, \mathfrak{h}_\C) = (\pm (e_i + e_j), 1 \leqslant i < j \leqslant k ). $$
\begin{Lemma}\cite[(5.2)]{SW}\label{uranus}
There exists $M_\gamma \in \R \setminus \{0\}$ such that
$$E_\gamma(h_\varphi) = M_\gamma \cdot \lim_{\gamma_\varepsilon \to \gamma} \prod_{\alpha \in \Delta^+_\gamma} H_\alpha E_{\gamma_\varepsilon}(h_\varphi).$$
\end{Lemma}
We are ready to prove the main theorem in this subsection:
\begin{Th}\label{saturn} There exists an even polynomial $P^\gamma_\sigma(i \nu)$ such that 
$$E_\gamma(h_\varphi) = \sum_{\sigma \in \widehat{SO(2n)}} \int_{\R} P^\gamma_\sigma (i \nu) \Theta_{\sigma,\nu} (h_\varphi) d \nu. $$
\end{Th}
\begin{proof}
Theorem \ref{saturn} holds with 
\begin{equation}\label{willstdubisdertod}
P^\gamma_\sigma(i \nu) = \prod_{\alpha \in \Delta^+_\gamma} H_\alpha \left( \sum_{s \in W} \det (s)
\left( \xi_{-s(\Lambda(\sigma) + \delta_M)- i e_1 \nu}(\gamma_\varepsilon) \right) \right)
\end{equation}
by Lemmas  \ref{neptune} and \ref{uranus}. We need to show that $P^\gamma_\sigma(i \nu)$ is an even polynomial. Note that
\begin{equation}\label{hellokitty}
\begin{gathered}
\prod_{\alpha \in \Delta^+_\gamma} H_\alpha \left( \xi_{-s(\Lambda(\sigma) + \delta_M)- i e_1 \nu} \right)= 
\left\lbrace  \prod_{\alpha \in \Delta^+_\gamma} \langle-s(\Lambda + \delta_M) - i \nu e_1, \alpha\rangle \right\rbrace \left( \xi_{-s(\Lambda(\sigma) + \delta_M)- i e_1 \nu} \right).
\end{gathered}
\end{equation}
Let $s(\Lambda + \delta_M) = \sum_{2 \leqslant i \leqslant n+1} k_i e_i$ with $\delta_M$ as in (\ref{halfsummy}). Then
\begin{equation}\label{mars}
\begin{gathered}
\prod_{\alpha \in \Delta^+_\gamma} \langle -s(\Lambda + \delta_M) - i \nu e_{1}, \alpha \rangle =\\
(-1)^{|\Delta_\gamma^+|} \prod_{1 \leqslant i' < j' \leqslant k} \langle\sum_{2 \leqslant i \leqslant n+1} k_i e_i + i \nu e_{1}, e_{i'} - e_{j'}\rangle \cdot \langle\sum_{2 \leqslant i \leqslant n+1} k_i e_i + i \nu e_{1}, e_{i'} + e_{j'}\rangle =\\
(-1)^{|\Delta_\gamma^+|} \prod_{1 \leqslant i' < j' \leqslant k}  \left( i \nu (\delta_{i', 1} + \delta_{j', 1}) + (k_{i'}+k_{j'})  \right)\cdot  \left( i \nu (\delta_{i', 1} - \delta_{j',1}) + (k_{i'}-k_{j'})  \right) =\\
(-1)^{|\Delta_\gamma^+|} \prod_{1 \leqslant i' < k} \left(i \nu + k_{i'} \right) \cdot \left( -i \nu + k_{i'}  \right) =
(-1)^{|\Delta_\gamma^+|} \prod_{1 \leqslant i' < k} \left(\nu^2 + k_{i'}^2  \right)
\end{gathered}
\end{equation}
Note that (\ref{mars}) is an even polynomial in $\nu$ and by Remark \ref{inspirationpublic} the character $\xi_{-s(\Lambda(\sigma) + \delta_M)- i e_1 \nu}$ does not depend on $\nu$. Hence, (\ref{hellokitty}) and (\ref{willstdubisdertod}) are even polynomials in $\nu$ as well.  
\end{proof}

We would like to mention the resemblance of Theorem \ref{saturn} to the following:
\begin{Proposition}\cite[Theorem 13.2]{Kna} \label{plashka}
There exists an even polynomial $P_{\sigma'}(i \nu)$ such that
$$\tr h_\varphi(e) = \sum_{\sigma' \in \widehat{M}} \int_{\R} P_{\sigma'}(i \nu) \Theta_{\sigma, \nu} (h_\varphi) d \nu.$$
\end{Proposition}

For further use we need to show one property of the polynomial $P_\sigma^\gamma(i \nu)$. Let $M'$ be the normalizer of $A$ in $K$ and let $W(A)=M'/M$ be the restricted Weyl group. It has order 2 and acts on finite-dimensional representations of $M$ \cite[p.~18]{Pf}. Let $\sigma$ be a finite-dimensional representation of $M$  with the highest weight
\begin{equation}\label{verzauertwennihrblick}
\Lambda(\sigma) = \sum_{j=2}^{n+1} \lambda_j(\sigma) e_j,
\end{equation}
then the highest weight of a representation $w_0 \sigma$, where $w_0$ is the non-identity element of $W(A)$, equals
\begin{equation}\label{dichtrifft}
 \Lambda(w_0 \sigma) = \sum_{j=2}^{n} \lambda_j(\sigma) e_j - \lambda_{n+1}(\sigma) e_{n+1}.
\end{equation}
\begin{Lemma}\label{feuerundwasser}
The polynomial $P_\sigma^{\gamma}$ is invariant under the action of $W(A)$: $$P^\gamma_\sigma(i \nu) = P^\gamma_{w_0 \sigma}(i \nu).$$
\end{Lemma}
\begin{proof}
Recall that $s \in W$ acts on the roots by even sign changes and the permutations. Then it follows from (\ref{verzauertwennihrblick}) and (\ref{dichtrifft}) that if $s(\Lambda + \delta_M) = \sum_{2 \leqslant i \leqslant n+1} k_i e_i$, then $s(\Lambda + \delta_M) = \sum_{2 \leqslant i \leqslant n+1} \hat{k}_i e_i$,
where $\hat{k}_i = - k_i$ for exactly one $i$ and $\hat{k}_j = k_j$ for all $j \neq i$. It follows that $\hat{k}_i^2 = k_i^2$. By (\ref{mars}) the polynomial $P_\sigma^\gamma(i \nu)$ depends only on $k_i^2$ which completes the proof of Lemma \ref{feuerundwasser}.
\end{proof}

\section{Selberg zeta function}\label{sec:SZF}

\begin{definition}\label{thermos}
The Selberg zeta function is:
$$ Z(s, \sigma, \chi) := \exp\left( - \sum_{\{ \gamma\} \, \text{hyperbolic}} \frac{\tr(\chi(\gamma)) \cdot v(\gamma) \cdot \Tr (\sigma(m_\gamma)    
\cdot e^{-(s+n) l(\gamma)})  }{n_{\Gamma'}(\gamma)    \det(\Id - \Ad(m_\gamma a_\gamma)|_{\mathfrak{n}})  }\right).$$
\end{definition}
\begin{remark}
Definition \ref{thermos} differs from  \cite{BO} by the term $v(\gamma)$, that equals $1$ for manifolds. 
\end{remark}

\begin{Proposition}\label{tarantellada}
 There exist $c>0$ such that $Z(s, \sigma)$ converges absolutely and locally uniformly for $\RRe(s)>c$.
\end{Proposition}\label{conviiii}
\begin{proof}
 Recall that $\Gamma'$ is of finite index in $\Gamma$, hence $v(\gamma) = \frac{\vol (\Gamma_\gamma \bs G_\gamma)}{\vol(\Gamma'_\gamma \bs G_\gamma)}$ is bounded in $\gamma$. On the other hand, by \cite[Lemma 3.3]{Spil} there exists $k,K \ge 0$ such that  
 $$ \tr(\chi(\gamma)) \le K e^{k l(\gamma)}.$$
The rest of the proof follows according to \cite{BO}. 
\end{proof}

\subsection{The symmetric Selberg zeta function.}\label{sec:symmsel}
Let $\sigma \in \widehat{SO(2n)}$. For $\RRe(s) > c$ with the constant $c$ as in Proposition \ref{tarantellada} we define the symmetric Selberg zeta function by 
\begin{equation}
    S(s, \sigma) = 
    \begin{cases}
      Z(s, \sigma) Z(s, w_0 \sigma), & \text{if}\ \sigma \neq w_0 \sigma; \\
      Z(s, \sigma) , & \text{if}\ \sigma = w_0 \sigma.
    \end{cases}
\end{equation}

In this subsection we prove the existence of the meromorphic continuation of the symmetric Selberg zeta function. We follow the approach of \cite{Pf} which associates a vector bundle $E(\sigma)$ to every representation $\sigma \in \widehat{SO(2n)}$. This vector bundle is graded  and there exists a canonical graded differential operator $A(\sigma)$ which acts 
on smooth sections of $E(\sigma)$. The next step is to apply the Selberg trace formula with a certain test function.

First, we construct the bundle $E(\sigma)$ and the operator $A(\sigma)$. By \cite[Prop.~2.12]{Pf}, there exist integers $m_\nu(\sigma) \in \{ -1, 0, 1\}$ such that for $\sigma = w_0 \sigma$ one has
\[
\sigma = \sum_{\nu \in \widehat{K}} m_\nu(\sigma) \iota^* \nu
\]
and for $\sigma \neq w_0 \sigma$ one has
\[
\sigma + w_0 \sigma = \sum_{\nu \in \widehat{K}} m_\nu(\sigma) \iota^* \nu.
\]
Above $\iota^*: R(K) \to R(M)$  the restriction map induced by the inclusion $\iota: M \hookrightarrow K$, where $R(K)$ and $R(M)$ are the representation rings over $\Z$ of $K$ and $M$, respectively. Moreover, $m_\nu(\sigma)$ are zero except for finitely many $\nu \in \widehat{K}$.  
Let $E_{\sigma', \rho}$ be the orbibundle associated to $\sigma' \in \widehat{K}$, $\rho: \Gamma \to GL(V)$ as in Subsection \ref{Sec9}.
Let $E(\sigma)$ be the orbibundle 
$$
E(\sigma) := \bigoplus_{\nu: \, m_\nu(\sigma) \neq 0} E_{\nu, \rho}.
$$
Note that $E(\sigma)$ has a grading $E(\sigma) = E^+ \oplus E^-$ defined by the sign of $m_\nu(\sigma)$.
For every $\nu \in \widehat{K}$ let $A_{\nu, \rho}$ be the operator defined by
\[
A_{\nu, \rho} := \Delta_{\nu, \rho}^\# + c(\sigma) - \tau(\Omega_K),
\]
where $\Delta_{\nu,\rho}^\#$ is as in Subsection \ref{Sec8}, $c(\sigma)$ is as in \cite[(2.27)]{Pf}, $\tau(\Omega_K)$ is as in (\ref{20}). Let $A(\sigma)$ be the operator acting on $C^\infty(\Ob, E(\sigma))$
defined by 
$$
A(\sigma)  := \bigoplus_{\nu: \, m_\nu(\sigma) \neq 0} A_{\nu,\rho}.
$$
Let $\tilde{E}(\sigma):=\bigoplus_{\nu:\, m_\nu(\sigma) \neq 0} \tilde{E}_{\nu, \rho}$ be the lift of $E(\sigma)$ to $\HH^{2n+1}$, and let $\tilde{A}(\sigma)$ be the lift of $A(\sigma)$  to $\tilde{E}(\sigma)$. Note that by (\ref{princesssweety})
$$ \tilde{A}(\sigma) = \bigoplus_{\nu: \, m_\nu(\sigma) \neq 0} \tilde{\Delta}_{\nu} + c(\sigma) - \tau(\Omega_k) \otimes \Id.$$
Together with (\ref{20}) it gives
$$ \tilde{A}(\sigma) = \bigoplus_{\nu: \, m_\nu(\sigma) \neq 0} -\Omega(R)+c(\sigma) \otimes \Id.$$
Second, we wish to apply the Selberg trace formula to $A(\sigma)$. For this let
\begin{equation}
h_t^\sigma(g) := \sum_{\nu: \, m_\nu(\sigma) \neq 0} m_\nu(\sigma) h_t^\nu (g),
\end{equation}
where $h_t^\nu:=\tr H_t^\nu$, and $H_t^\nu$ is the integral kernel of $e^{-t \tilde{\Delta}_\nu}$.
Proposition~\ref{propSTF}, Lemma~\ref{moyaaktrisa}, Theorem~\ref{saturn} and Proposition~\ref{plashka} imply:
\begin{Th}\label{supermassivblackhole}
The supertrace of $e^{-t A(\sigma)}$, taken with respect to the grading defined by the sign of $m_\nu(\sigma)$, equals
\begin{equation*}\label{functional_equation_selberg_zeta}
\begin{gathered}
\text{Tr}_s(e^{-t A(\sigma)}) = \vol(X) \dim(V_\chi)\sum_{\sigma' \in \hat{M}} \int_\R P_\sigma (i \lambda) \Theta_{\sigma', \lambda} (h_t^\sigma) d\lambda + \\
\sum_{\sigma' \in \hat{M}} \sum_{[\gamma] \, \textnormal{elliptic}} \vol(\Gamma_\gamma \bs G_\gamma) \tr (\chi(\gamma)) \sum_{\sigma' \in \hat{M}} \int_\R P^\gamma_\sigma(i \lambda) \Theta_{\sigma', \lambda} (h_t^\sigma) d \lambda +\\
 \sum_{\sigma' \in \hat{M}} \sum_{[\gamma] \, \textnormal{hyperbolic}}  \frac{\tr (\chi(\gamma)) \, v(\gamma)\, l(\gamma_0)}{2 \pi D(\gamma)} \overline{\tr(\sigma'(\gamma))}  \int_\R \Theta_{\sigma', \lambda} (h_t^\sigma) e ^{- l(\gamma) \lambda}d\lambda.
\end{gathered}
\end{equation*}
\end{Th}

\begin{Lemma}\cite[Section 4]{MP}\label{aerozolka}
$\Theta_{\sigma', \lambda}(h_t^\sigma) = e^{-t \lambda^2}$ for $\sigma' \in \{ \sigma, w_0 \sigma\}$ and equals zero otherwise.
\end{Lemma}
Let

\begin{equation}\label{zara}
    \epsilon(\sigma) = 
    \begin{cases}
      2, & \text{if}\ \sigma \neq w_0 \sigma; \\
      1 , & \text{if}\ \sigma = w_0 \sigma.
    \end{cases}
\end{equation}
Denote
\begin{equation}\label{sosprach}
\begin{gathered}
I(t) := \epsilon(\sigma) \dim (V_\chi) \vol(\Ob) \int_{\R} P_{\sigma}(i \lambda) e^{-t \lambda^2} dt,\\
E(t) := \epsilon(\sigma) \sum_{\{\gamma\}  \, \text{elliptic}} \tr (\chi(\gamma)) \vol(\Gamma_\gamma \bs G_\gamma) \int_{\R} P^\gamma_\sigma(i \lambda) e^{-t \lambda^2} dt,\\
H(t) := popozzhe opredelyu.
\end{gathered}
\end{equation}
Then Lemma~\ref{feuerundwasser}, Theorem \ref{supermassivblackhole}, Lemma \ref{aerozolka} together with (\ref{zara}) and (\ref{sosprach}) imply
\begin{equation}
\text{Tr}_s (e^{-t A(\sigma)}) = I(t) + H(t) + E(t).
\end{equation}
Denote $(A(\sigma)+s^2)^{-1} = R(s^2)$. Note that
$$R(s^2)=\int_o^\infty e^{-ts^2}e^{-t A(\sigma)} dt.$$
The operator $R(s^2)$ is not a trace class operator, but we will now improve it. 
\begin{Lemma}\cite[Lemma 3.5]{BO}\label{sorgensorgen}
Let $s_1, \ldots , s_N \in \C$  such that 
$s_i^2 \neq s_j^2$ for $i \neq j$. Then for every 
$z \in \C \setminus \{ -s_1^2, \ldots, -s_N^2 \}$ one has
\begin{equation*}
\sum_{i=1}^N \frac{1}{s_i^2 + z} \prod_{j = 1, j\neq i}^{N} \frac{1}{s_j^2 - s_i^2}= \prod_{i=1}^N \frac{1}{s_i^2+z},
\end{equation*}
hence, for $c_i = \prod_{j = 1, j\neq i}^{N} \frac{1}{s_j^2 - s_i^2}$, $i=1, \ldots, N,$ 
\begin{equation}\label{yoyouarepseudo}
\sum_{j=1}^N c_j R(s_j^2) = \prod_{j=1}^N R(s_j^2).
\end{equation}
\end{Lemma}
\begin{Lemma}
The operator $\prod_{j=1}^N R(s_j^2)$ is of trace class.
\end{Lemma}
\begin{proof}
In \cite{BO} Lemma \ref{sorgensorgen} was proven for manifolds by the following argument: each of the factors is a pseudodifferential operator of order $-2/(2n+1)$, hence their product is a pseudodifferential operator of order $-2 N / (2n+1)$ that is of trace class for sufficiently large $N$ by the Weyl law. In our setting we have established the Weyl law only for a non-selfadjoint Laplacian, so we need to make some additional effort. 
\begin{Lemma}
The operator $R(s_1^2)^{N}$ is of trace class for  $N > (2n+1)/2$.
\end{Lemma}
\begin{proof}
By Theorem \ref{darthveider}, the $k$-th eigenvalue of $R(s_1^2)$ has order $k^{-2/(2n+1)}$, hence the $k$-th eigenvalue of  $R(s_1^2)^{N}$ has order $k^{-2N/(2n+1)}$.
\end{proof}
\begin{Lemma}
The operator $R(s_1^2)^{-N}\cdot \prod_{j=1}^N R(s_j^2)$ is bounded.
\end{Lemma}
\begin{proof}
It is a pseudodifferential operator of order 0.
\end{proof}
By the above two lemmas,
$$ \prod_{j=1}^N R(s_j^2) = R(s_1^2)^{N} \cdot \left( R(s_1^2)^{-N}\cdot \prod_{j=1}^N R(s_j^2)\right)$$
is of trace class.
\end{proof}
Then
\begin{equation}\label{expr1sd}
\Trs \left(  R(s^2)+\sum_{j=1}^N c_j R(c_j^2)\right) =\int_0^\infty \left(e^{-ts^2} + \sum_{j=1}^{N} c_j e^{-ts_j^2} \right)\cdot \Trs (e^{-t A(\sigma)}) dt.
\end{equation}
We would like to apply Theorem \ref{supermassivblackhole} to the right hand side. By analogy with  \cite{Pf},
\begin{equation} 
\begin{gathered}
\int_0^\infty \left(e^{-ts^2} + \sum_j c_j e^{-ts_j^2}\right) \cdot H(t) dt = \frac{1}{2s} \frac{S'(s,\sigma)}{S(s,\sigma)} + \sum_j \frac{c_j}{2 s_j} \frac{S'(s_j,\sigma)}{S(s_j,\sigma)}, \\
\int_0^\infty \left(e^{-ts^2} + \sum_j c_j e^{-ts_j^2}\right)\cdot I(t) dt = \epsilon(\sigma) \vol(X) \dim(V_\chi) \cdot \left( \frac{\pi}{s} P_\sigma(s) + \sum_j \frac{c_j \pi}{s_j} P_\sigma (s_j) \right),\\
\int_0^\infty \left(e^{-ts^2} + \sum_j c_j e^{-ts_j^2}\right)\cdot E(t) dt = \sum_{ \{\gamma\} \, \text{elliptic}}\epsilon(\sigma) \vol(\Gamma_\gamma \bs G_\gamma) \tr (\chi(\gamma))\cdot \left(  \frac{\pi}{s} P_{\sigma}^\gamma (s) + \sum_j \frac{c_j \pi}{s_j} P_{\sigma}^\gamma (s_j) \right).
\end{gathered}
\end{equation}
Note that we are crucially using that $P_{\sigma}^\gamma(\nu)$ and $P_\sigma(\nu)$ are even polynomials in $\nu$. Thus we get
\begin{equation*}\label{preDF}
\begin{gathered}
\Trs\big(R(s^2)+\sum_j c_j R(c_j^2)\big) = \frac{1}{2s} \frac{S'(s,\sigma)}{S(s,\sigma)} + \sum_j \frac{c_j}{2 s_j} \frac{S'(s_j,\sigma)}{S(s_j,\sigma)}  +\\  \epsilon(\sigma) \dim(V_\chi) \vol(\Ob) \cdot \left(\frac{\pi}{s} P_\sigma(s) + \sum_j \frac{c_j \pi}{s_j} P_\sigma (s_j) \right)+\\
 \sum_{ \{\gamma\} \text{ elliptic}}\epsilon(\sigma) \vol(\Gamma_\gamma \bs G_\gamma) \tr (\chi(\gamma)) \cdot \left(  \frac{\pi}{s} P_{\sigma}^\gamma (s) + \sum_j \frac{c_j \pi}{s_j} P_{\sigma)}^\gamma (s_j) \right).
\end{gathered}
\end{equation*}
Put 
\begin{equation}\label{zachemuhodit}
\Xi(s,\sigma) = \exp\left(- 2 \pi \epsilon(\sigma)\dim(V_\chi) \vol(X) \int_0^s P_\sigma(r) dr - 2 \epsilon(\sigma)  \sum_{\{ \gamma \} \text{ elliptic}} \tr (\chi(\gamma)) \int_0^s P^\gamma_{\sigma} (r) dr\right) \cdot S(s, \sigma)
\end{equation}
Then (\ref{preDF}) can be rewritten as 
\begin{equation}\label{vnoch}
\Trs \left( R(s^2)+\sum_{j=1}^N c_j R(c_j^2)  \right) = \frac{1}{2s} \frac{\Xi'(s,\sigma)}{\Xi(s,\sigma)}+\sum_{j=1}^N \frac{c_j}{2s_j}\frac{\Xi'(s_j,\sigma)}{\Xi(s_j, \sigma)}.
\end{equation}
From (\ref{zachemuhodit}) and (\ref{vnoch}) one can deduce the existence of the meromorphic extension of $S(s,\sigma)$ and determine the location of its singularities. Let $\lambda_1 < \lambda_2 < \ldots$ be the eigenvalues
of $A(\sigma)$. For each $\lambda_j$ let $\mathcal{E}(\lambda_j)$ be the eigenspace of $A(\sigma)$ with eigenvalue $\lambda_j$. Put
$$ m_s(\lambda_j, \sigma)=\dim_{gr} \mathcal{E}(\lambda_j),$$
where by $\dim_{gr}$ we denote the draded dimension.
If $\lambda_j <0$, we choose the square root $\sqrt{\lambda_j}$ which has positive imaginary part. Put
$$s_j^\pm = \pm i \sqrt{\lambda_j},  j \in \N.$$
Note that by  (\ref{zachemuhodit}) the symmetric Selberg zeta function $S(s,\sigma)$ admits a meromorphic extension to~$\C$ if and only if $\Xi(s,\sigma)$ does. In order to prove that $\Xi(s,\sigma)$ admits a meromorphic extension to $\C$ it is sufficient to show that all  residues of $2s \cdot \Trs\left( R(s^2)+\sum_{j=1}^N c_j R(c_j^2)  \right)$ are integers. 
\begin{Th}
The symmetrised Selberg zeta function  $S(\sigma,s)$ has a meromorphic extension to $\C$. The set of singularities of $S(s,\sigma)$ equals $\{ s^{\pm}_j: j \in \N \}$. If $\lambda_j \neq 0$, then the order
of $S(s,\sigma)$ at both $s_j^+$ and $s_j^-$ is equal to $m_s(\lambda, \sigma)$. The order of the singularity at $s=0$ is $2m_s(0,\sigma)$.
\end{Th}

\subsection{Antisymmetric Selberg zeta function.}\label{sec:aszf}

Suppose that $\sigma \neq w_0 \sigma$, otherwise the  symmetric Selberg zeta function equals the Selberg zeta function and this section can be skipped.
For $\RRe(s) > c$ with the constant $c$ as in Proposition \ref{tarantellada} we define the antisymmetric Selberg zeta function as 
\begin{equation}\label{ultramassive}
    S_a(s, \sigma) :=  Z(s, \sigma) / Z(s, w_0 \sigma).
\end{equation}
In this subsection we prove the meromorphic continuation of antisymmetric Selberg zeta function $S_a(s, \sigma)$. 
For this we need some additional constructions. We introduce  Dirac operators $\tilde{D}(\sigma)$ and $D$ on $\HH^{2n+1}$. 
They are $G$-invariance and hence descend to $D(\sigma)$ on the orbifold $\Ob$.
Recall that in the clasical case (unitary twists on manifolds) the antisymmetric Selberg zeta function $S_a(s, \sigma)$ 
can be studied with the help of Selberg trace formula. Namely, the hyperbolic contribution to $\Trs D e^{- t D^2}$
equals the logarithmic derivative of $S_a(s, \sigma)$. 
We cannot apply Selberg trace formula in the exact form as it is stated in Proposition \ref{propSTF}  to study $\Tr D e^{- t D^2}$, because $D(\sigma) e^{- t D^2(\sigma)}$ cannot be obtained as an even function of $D(\sigma)$. Therefore we need a modified version of the Selberg trace formula.

\subsection*{Dirac bundles and twisted Dirac operators}

Let $\Cl(\mathfrak{p})$ be the Clifford algebra of $\mathfrak{p}$ with respect to the scalar product on $\mathfrak{p}$. Let $\Cl(\HH^{2n+1}) := G\times_{\Ad} \Cl(\mathfrak{p})$  be the Clifford bundle over $\HH^{2n+1}$.
Let $\tilde{S} = G \times_\kappa \Delta^{2n}$ be the spinor bundle on $\HH^{2n+1}$. We denote by $c: \Cl(\mathfrak{p})\otimes \Delta^{2n}~\to~\Delta^{2n}$, $X\otimes v \mapsto   c(X) v$ the Clifford multiplication.
This multiplication induces naturally a Clifford multiplication $\Cl(\HH^{2n+1}) \otimes \tilde{S} \to \tilde{S}$. Since $M$ centralizes $\mathfrak{a}$, there isx $\epsilon \in \{ \pm 1\}$ such that $\epsilon c(H_1)$
acts on the spaces $\Delta_\pm^{2n}$ with eigenvalues $\mp i$.

Let $\tilde{E}(\sigma)$ be the vector bundle over $\HH^{2n+1}$, as in Section \ref{sec:symmsel}. Note that $\nu(\sigma) \otimes \kappa = \nu^+(\sigma) \oplus \nu^-(\sigma)$ and $\tilde{E}(\sigma) = \tilde{E}_{\nu^+(\sigma)} \oplus \tilde{E}_{\nu^-(\sigma)}$ with $\nu(\sigma) \in \hat{K}$, $\kappa$ is the spin-representation of $K$ as in \cite[{Proposition 1.1}]{BO}.
Together with $\tilde{S} = \tilde{E}_\kappa$ it gives a splitting $\tilde{E}(\sigma) = \tilde{E}_{\nu(\sigma)} \otimes \tilde{S}$.
Let $\Cl(\HH^{2n+1})$ act on $\tilde{E}(\sigma)$ by tensoring the trivial action of $\Cl(\HH^{2n+1})$ on $\tilde{E}_{\nu(\sigma)}$ with the action of $\Cl(\HH^{2n+1})$
just defined.

Together with the metrics and connections chosen as in \cite[p. 24]{Pf1}, $\tilde{E}(\sigma)$ with the Clifford action defined above
becomes a Dirac bundle in the sense of \cite[Chapter II, Definition 5.2]{LM}. Thus we can introduce the associated Dirac operator $\tilde{D}(\sigma)$ acting on sections of $\tilde{E}(\sigma)$.
The construction is $G$-invariant, so the operator $\tilde{D}(\sigma)$ descends to the new operator $D(\sigma)$ on $\Ob$, that can be expressed as
\begin{equation}\label{def:tdo}
 D(\sigma) = \sum_i e_i \cdot \nabla^E_{e_i}.
\end{equation}
For more details on the construction we refer to \cite{BO} or \cite{Pf}. Analogously to the previous sections,
consider the orbibundle $F \to \Ob$ with the flat connection $\nabla^F$ associated with the representation $\rho: \Gamma \to V$. Recall the following fact:
\begin{Proposition}\cite[Chapter II, Proposition 5.10]{LM}\label{pr:orbtw}
Let $E$ be any Dirac bundle over a Riemannian manifold X. Suppose $F$ is any Riemannian bundle with connection. Then the tensor 
product $E \otimes F$ is again a Dirac bundle over $X$.
\end{Proposition}
Note that Proposition \ref{pr:orbtw} is valid for orbifolds. Define a Dirac operator $D$ on $E(\sigma)\otimes F$:
\begin{equation}\label{def:ttdo}D := \sum_i e_i \cdot \nabla^{E(\sigma) \otimes F}_{e_i},\end{equation}
where $\nabla^{E \otimes F} = 1 \otimes \nabla^F + \nabla^{E(\sigma)} \otimes 1$. Note that $D^2$ a second order elliptic differential operator and by a version of Theorem \ref{ShuEst}, its spectrum is discrete and there exist $R \in \R$ and $\varepsilon > 0$ such that  
\begin{equation}\label{contourmakaka}
\spec(D^2) \in L := \Lambda_{[-\varepsilon,\varepsilon]} \cup B(R).
\end{equation}

\subsection*{Selberg trace formula}\label{sec:fa}
In this subsection we verify that the Selberg trace formula can be applied to the operator $D e^{-t D^2}$. We define the operator $De^{-tD^2}$
via the integral
\begin{equation}\label{eq:edirac}
D e^{-t D^2} := \frac{i}{2 \pi}\int_B e^{-t \lambda} D (D^2 - \lambda)^{-1} d \lambda
\end{equation}
for $B = \partial L$ with $L$ as in (\ref{contourmakaka}). 
\begin{Proposition}
The right hand side of (\ref{eq:edirac}) converges.
\end{Proposition}

\begin{proof}
Follows from Theorem~\ref{spectrBehaviour}.
\end{proof}
Note that the lift of $D$ to $\HH^{2n+1}$ splits into $\widetilde{D}(\sigma) \otimes \Id$ by the same arguments as in Sections~\ref{Sec8}~and~\ref{sec:symmsel}, where $\widetilde{D}(\sigma)$ is a lift of $D(\sigma)$. Also the operator $D e^{-t D}$ is an integral operator with smooth kernel, because $e^{-tD}$ is. By an analogy with the previous calculations we obtain:
\begin{Lemma}
Denote by $k_t^\sigma(\cdot)$
the convolution kernel of $\widetilde{D}(\sigma) e^{-t \widetilde{D}^2(\sigma) } $. Then we have
\begin{equation}\label{eq:preSel}
\sum_{\lambda \in \spec(D)} \lambda e^{- t \lambda^2} = \dim(V_\chi) \vol (\Gamma \backslash S) \tr k_t^\sigma(e) + \\
 \sum_{\{\gamma\} \neq \{e\} } \tr \chi(\gamma) \vol (\Gamma_\gamma \backslash G_\gamma)  \int_{G_\gamma \backslash G} \tr k_t^\sigma(g^{-1} \gamma g) d\dot{g}.
\end{equation}
\end{Lemma}
As an analogue of Theorem \ref{supermassivblackhole} we get
\begin{equation}\label{eq:sptrtrtrt}
\begin{gathered}
\text{Tr}_s(D e^{-t D^2}) = \vol(X)\dim(V_\chi) \sum_{\sigma' \in \hat{M}} \int_\R P_\sigma (i \lambda) \Theta_{\sigma', \lambda} (k_t^\sigma) d\lambda + \\
\sum_{\sigma' \in \hat{M}} \sum_{[\gamma] \text{ elliptic}} \vol(\Gamma_\gamma \bs G_\gamma) \tr (\chi(\gamma)) \sum_{\sigma' \in \hat{M}} \int_\R P^\gamma_\sigma(i \lambda) \Theta_{\sigma', \lambda} (k_t^\sigma) d \lambda +\\
 \sum_{\sigma' \in \hat{M}} \sum_{[\gamma] \text{ hyperbolic}}  \frac{\tr (\chi(\gamma)) \, v(\gamma)\, l(\gamma_0)}{2 \pi D(\gamma)} \overline{\tr(\sigma'(\gamma))}  \int_\R \Theta_{\sigma', \lambda} (k_t^\sigma) e ^{- l(\gamma) \lambda}d\lambda.
\end{gathered}
\end{equation}

\begin{Proposition}\label{pr:vanish}\cite[Proposition 8.2]{Pf1}, \cite{MS}
Let $\sigma \in \hat{M}$, $k_{n+1}(\sigma)>0$. Then for $\lambda \in \R$ one has
$$ \Theta_{\sigma, \lambda} (k) = (-1)^n \lambda e^{-t\lambda^2}, \quad \Theta_{w_0 \sigma, \lambda} (k) = (-1)^{n+1} \lambda e^{-t\lambda^2}.$$
Moreover, if $\sigma' \in \hat{M}, \sigma' \neq \{ \sigma, w_0 \sigma\},$ for every $\lambda \in \R$ one has $\Theta_{ \sigma', \lambda}(k) = 0.$
\end{Proposition}
Applying Proposition \ref{pr:vanish} to (\ref{eq:sptrtrtrt}), we get 
\begin{equation}\label{bananana}
\begin{gathered}
(-1)^n \Trs(D e^{-t D^2}) = \vol(X) \dim(V_\chi)\int_\R (P_\sigma (i \lambda) - P_{w_0 \sigma} (i \lambda)) \lambda e^{- t\lambda^2}d\lambda + \\  + \sum_{[\gamma] \text{ elliptic}}   \vol(\Gamma_\gamma \bs G_\gamma)  \int_\R (P^\gamma_\sigma(i \lambda) - P^\gamma_{w_0 \sigma}(i \lambda)) \lambda e^{- t\lambda^2} d \lambda +\\
+  \sum_{[\gamma] \text{ hyperbolic}} C_2(\gamma) \frac{l(\gamma_0)}{2 \pi} \left( L(\gamma, \sigma) - L(\gamma, w_0 \sigma) \right) \int_\R \lambda e^{-t \lambda^2} e ^{- l(\gamma) \lambda}d\lambda,
\end{gathered}
\end{equation}
Moreover, the first and the second summand in the right hand side of  (\ref{bananana}) vanish by the following two remarks.
\begin{remark}
By \cite[(2.22)]{MP}
$$P_\sigma(i \nu) - P_{w_0 \sigma}(i \nu) = 0.$$
\end{remark}
\begin{remark} By Lemma \ref{feuerundwasser}
$$P^\gamma_\sigma(i \nu) - P^\gamma_{w_0 \sigma}(i \nu) = 0.$$ 
\end{remark} 

We proceed as in Section \ref{sec:symmsel}. The operator $D \cdot (D^2 + s^2)^{-1}$ is not of trace class, but we can choose coefficients $c_j$ and $s_j$ such that $D \cdot (D^2 + s^2)^{-1} + \sum_j c_j D \cdot (D^2 + s_j^2)^{-1}$
is of trace class. By the same arguments as in (\ref{expr1sd})-(\ref{preDF}) and the vanishing of $P_{w_0 \sigma}^\gamma - P_\sigma^\gamma$ and $P_{w_0 \sigma} - P_{\sigma}$ we obtain
$$ \Tr \big(   D \cdot (D^2 + s^2)^{-1} + \sum_j c_j D \cdot (D^2 + s_j^2)^{-1} \big) = \frac{1}{2s} \frac{S_a'(s, \sigma)}{S_a(s, \sigma)}  + \sum_j \frac{c_j}{2 s_j} \frac{S_a'(s_j, \sigma)}{S_a(s_j, \sigma)}$$

The theorem below follows. 
\begin{Th}
The antisymmetric Selberg zeta $S_a(s, \sigma, \chi)$ function has a meromorphic extension to $\C$. It has singularities at the points $\pm i \mu_k$ of order $\frac{1}{2} (d(\pm \mu_k, \sigma) - d(\mp \mu_k, \sigma)),$ where $\mu_k$ is a non-zero eigenvalue of $D$ of multiplicity $d(\mu_k, \sigma)$. 
\end{Th}
Using that $Z(s, \sigma) = S(s, \sigma) S_a(s, \sigma)$, we obtain

\begin{Th}
The Selberg zeta function has an meromorphic extension to $\C$. It has the following singularities:
\begin{itemize}
\item If $\sigma = w_0 \sigma$, a sigularity at the points $\pm i \sqrt{\lambda_k}$ of order $m_s(\lambda_k, \sigma)$, where $\lambda_k$ is a non-zero eigenvalue of $A(\sigma)$ and $m_s(\lambda_k, \sigma)$ is the graded dimension of the corresponding eigenspace.  
\item If $\sigma \neq w_0 \sigma$, a singularity at the points $\pm i \mu_k$ of order $\frac{1}{2} (m_s(\mu_k^2, \sigma) + d(\pm \mu_k, \sigma) - d(\mp \mu_k, \sigma)).$ Here $\mu_k$ is a non-zero eigenvalue of $D$ of multiplicity $d(\mu_k, \sigma)$ and $m_s(\mu_k^2, \sigma)$ is the graded dimension of the eigenspace $A(\sigma)$ corresponding to the eigenvalue $\mu_k^2.$
\item At the point $s=0$ a singularity of order $2m_s(0, \sigma)$ if $\sigma = w_0 \sigma$ and of order $m_s(0, \sigma)$ if $\sigma \neq w_0 \sigma.$
\end{itemize}
\end{Th}

\section{Appendix: Extensions of differential operators}
\label{Extens}
The goal of this section is to recall some issues that may arise when choosing a selfadjoint extension of
a Laplacian on a manifold with conical singularities. Another goal is to build a connection between such extensions and the
selfadjoint extension of a Laplacian on an orbifold. 
Let $\Gamma \bs \HH^2$ be a hyperbolic 2-dimensional orbifold with conical singularities $\{x_0, \ldots, x_n\}$. 
Suppose for a moment that a second order symmetric differential operator $D$ is defined
on~$C_0^\infty\left((\Gamma \bs \HH^2) \backslash \{x_0, \ldots, x_n\}\right)$, then $D$ is not essentially self-adjoint. We construct all its 
self-adjoint extensions $D_\alpha, \alpha \in \T^{n+1}$ and investigate the behaviour of functions
from the domain of $D_\alpha$ near the singularities  $\{x_0, \ldots, x_n\}$. We determine which extension of $D_\alpha$ 
corresponds to so-called orbifold extension $D_{orb}$, the one we  use throughout the article.
It is known that the orbifold extension of a symmetric operator is self-adjoint, see \cite{Buc}. We give a one-line proof of this statement
 under the condition that the orbifold is a global quotient orbifold $\Ob = \Gamma \bs M$ with $M$ a Riemannian manifold and $\Gamma$ a group of its isometries.

Consider an operator $\tilde{D}: C^\infty(M) \to C^\infty(M)$. For all $\gamma \in \Gamma$, let 
$L_\gamma: M \to M$, $L_\gamma(x) = \gamma(x)$. Suppose $\tilde{D}$ commutes with $L_\gamma$:
$$\tilde{D} \circ \Gamma_\gamma = \Gamma_\gamma \circ \tilde{D}, \quad \forall \gamma\in \Gamma.$$

Denote $C^\infty (\Gamma \bs M):= C^\infty(M)^\Gamma$. 
\begin{definition}
$D_{orb} := \tilde{D} |_{C^\infty(M)^\Gamma}$ is an orbifold extension of $D$.
\end{definition}
Suppose that $D \subset D_{orb}$.
\begin{remark}
It is possible to define an orbifold extension of $D$ for any orbifold, not necessarily good, see \cite{Buc}.
\end{remark}

A natural candidate for a differential operator $D$ that would be invariant under the isometries $M$ is the Laplacian. 
An orbifold extension of the Laplacian is easily proved to be essentially 
self-adjoint under some assumptions on the structure of $\Gamma$, namely, let 
$\Gamma' \subset \Gamma$ be a normal torsion free subgroup of finite index, $G = \Gamma / \Gamma'$.
 It follows that $\Gamma' \bs M$ is a complete Riemannian manifold. Hence, 
 $\D^{\Gamma' \bs M}: C^\infty(\Gamma' \bs M) \to C^\infty(\Gamma' \bs M)$ is an essentially self-adjoint operator  
 and $\overline{\D^{\Gamma' \bs M}}:H^2(M) \to L^2(M)$, hence 
 $\overline{\IIm (\D^{\Gamma' \bs M} \pm 1 )} = L^2(\Gamma' \bs M)$. Consider the orbifold extension 
 $\D^{\Gamma \bs M}$. The subspace of $G$-invariant functions are closed, and any Laplacian maps the space to itself, hence 
 $$\overline{\IIm (\D^{\Gamma \bs M} \pm 1 )} = \overline{\IIm (\D^{\Gamma' \bs M} \pm 1 \id)}^G = (L^2(\Gamma' \bs M))^G = L^2(\Gamma \bs M).$$

\begin{remark}
It is crucial that $G$ is finite. Otherwise, it would not be true that $(L^2(\Gamma' \bs M))^G = L^2(\Gamma \bs M)$.
\end{remark}

One can show, moreover, that any positive symmetric elliptic pseudodifferential operator on an orbibundle 
over a compact orbifold is essentially self-adjoint (\cite{Buc}, p. 37, Theorem 3.5). Note that the definition of a 
pseudodifferential operator on an orbifold requires that we consider an orbifold extension.

The fact that we only deal with self-adjoint orbifold extensions does not mean that no other self-adjoint extension 
of $\D$ exists. Here is an example when the space of all self-adjoint extensions of the Laplacian is 
a circle:
\begin{exmp} [Laplacians for manifolds with conical singularities]
Let $\mathbb{H}^2$ be an upper half-plane with the hyperbolic metric. Suppose that a group $G \in PSL(2,\R)$ 
has a cyclic element $\gamma$ and that $G \bs \mathbb{H} = \Ob$ is compact. Consider $\Ob$ not as an orbifold, but as a manifold with conical singularities. Then near a singular point 
$x_0$  corresponding to $\gamma$, $\Ob$ is isometric to a cone $(0,1]\times S^1$ with
the metric $ds^2 = dr^2 + \sinh^2(r) d\phi^2 / n^2$, where $r \in [0,1], \phi \in S^1$, and $n$ is the order of 
$\gamma$. Without loss of generality assume that there is exactly one singular point $x_0$. We will now describe all self-adjoint extensions of $\D$.
We follow \cite[p. 277-278]{YCdeV}, which treated the case of the  $\R^n \bs \{ 0 \}$
 with the flat metric. The only difference in our case is that sometimes one has to lift 
 to the covering  and descends to the quotient during the proof, and a rotation-invariant fundamental 
 solution $\D F = \delta(x_0)$ is not $F(x) = \log(x)$ as in the flat case, but $F(x) = \log(\coth(x/2))$. 
 In any case,  for $x$ sufficiently small we have $\log(\coth(x/2)) \sim \log(x)$, so the statement of theorem 
\cite[p. 227, Theorem 1]{YCdeV} stays true:
\begin{Th}
In the above example, all self-adjoint extensions of 
$\D$ can be parametrized by $\alpha \in \R / \pi \Z$, where the domain of $\D_\alpha$ is 
$$Dom(\Delta_\alpha) = \{ f \in Dom(\D_{max}) \, | \, \exists \lambda \in \C, f(x) = \lambda (\sin(\alpha)\log(x) + \cos(\alpha)) + o(1)\},$$
where $\D_{max}$ is the maximal extension of $\Delta$.
\end{Th}
\begin{remark}
The orbifold extension corresponds to $\alpha = 0$.
\end{remark}
\end{exmp}

If the operator $D$ we are dealing with does not admit a self-adjoint extension, sometimes it is still possible to classify all 
closed extensions of $\D$ between its minimal $\D_{min}$ and maximal $\D_{max}$ extensions, or at least to find the dimension of the 
space of closed extensions. For example, for first order elliptic operators  of a special kind, the classification was provided 
in~\cite[p. 672, Theorem 3.2]{BruSee}. More generally, the extensions of Fuchs type operators were studied in~\cite[Chapter 1]{Lesch}.

\bibliography{foo}{}

\begin{thebibliography}{DGGW08}

\bibitem[BGM71]{Ber}
M.~Berger, P.~Gauduchon, and E.~Mazet.
\newblock {\em Le spectre d'une vari{\'e}t{\'e} riemannienne}.
\newblock Springer, 1971.

\bibitem[BGV92]{Berl}
N.~Berline, E.~Getzler, and M.~Vergne.
\newblock {\em Heat {K}ernels and {D}irac {O}perators}.
\newblock Springer, 1992.

\bibitem[BO95]{BO}
U.~Bunke and M.~Olbrich.
\newblock {\em Selberg zeta and theta functions: a differential operator
  approach}.
\newblock Akademie Verlag, 1995.

\bibitem[BS88]{BruSee}
J.~Br{\"u}ning and R.~Seeley.
\newblock An index theorem for first order regular singular operators.
\newblock {\em Amer. J. Math}, 110(4):659--714, 1988.

\bibitem[Buc99]{Buc}
B.~Bucicovschi.
\newblock Seeley's theory of pseudodifferential operators on orbifolds.
\newblock {\em arXiv:math/9912228}, 1999.

\bibitem[DGGW08]{Dry}
E.~B. Dryden, C.~S. Gordon, S.~J. Greenwald, and D.~L. Webb.
\newblock Asymptotic expansion of the heat kernel for orbifolds.
\newblock {\em arXiv preprint arXiv:0805.3148}, 2008.

\bibitem[Don76]{Donnelly}
H.~Donnelly.
\newblock Spectrum and the fixed point sets of isometries. {I}.
\newblock {\em Math. Ann.}, 224(2):161--170, 1976.

\bibitem[dV82]{YCdeV}
Y.~C. de~Verdi{\`e}re.
\newblock Pseudo-laplaciens. {I}.
\newblock In {\em Annales de l'institut Fourier}, volume~32, pages 275--286.
  Institut Fourier, 1982.

\bibitem[Far01]{Far}
C.~Farsi.
\newblock Orbifold spectral theory.
\newblock {\em Rocky Mountain J. Math.}, 31(1):215--236, 2001.

\bibitem[Fed15a]{Fe3}
K.~Fedosova.
\newblock The analytic torsion for odd-dimensional finite volume hyperbolic
  orbifolds.
\newblock {\em in preparation}, 2015.

\bibitem[Fed15b]{Fe2}
K.~Fedosova.
\newblock On the asymptotics of the analytic torsion for compact hyperbolic
  orbifolds.
\newblock {\em arXiv:1511.04208}, 2015.

\bibitem[Fri05]{Fri}
J.~S. Friedman.
\newblock The {Selberg} trace formula and {S}elberg zeta-function for cofinite
  {Kleinian} groups with finite-dimensional unitary representations.
\newblock {\em Mathematische Zeitschrift}, 250(4):939--965, 2005.

\bibitem[GGPS68]{Gel}
I.~Gelfand, M.~Graev, and I.~Piatetski-Shapiro.
\newblock {\em Representation theory and automorphic functions}.
\newblock Saunders, 1968.

\bibitem[Gil95]{Gil}
P.B. Gilkey.
\newblock {\em Invariance theory, the heat equation, and the {A}tiyah-{S}inger
  index theorem. Second edition}.
\newblock Stud. Adv. Math., Boca Raton, FL, 1995.

\bibitem[GK69]{GK}
I.~Gohberg and M.~G. Krein.
\newblock {\em Introduction to the theory of linear nonselfadjoint operators},
  volume~18.
\newblock American Mathematical Soc., 1969.

\bibitem[GP10]{Par}
Y.~Gon and J.~Park.
\newblock {The zeta functions of Ruelle and Selberg for hyperbolic manifolds
  with cusps}.
\newblock {\em Mathematische Annalen}, 346(3):719--767, 2010.

\bibitem[Kna01]{Kna}
A.~W. Knapp.
\newblock {\em Representation theory of semisimple groups: An overview based on
  examples}, volume~36.
\newblock Princeton university press, 2001.

\bibitem[{Les}96]{Lesch}
M.~{Lesch}.
\newblock {Differential operators of Fuchs type, conical singularities, and
  asymptotic methods}.
\newblock In {\em eprint arXiv:dg-ga/9607005}, page 7005, July 1996.

\bibitem[LM90]{LM}
H.~Lawson and M.~Michelsohn.
\newblock {\em Spin Geometry}.
\newblock Princeton university press, 1990.

\bibitem[LR91]{LR}
J.~Lott and M.~Rothenberg.
\newblock Analytic torsion for group actions.
\newblock {\em J. Differential Geom.}, 34(2):431--481, 1991.

\bibitem[Mia80]{Mi}
R.~J. Miatello.
\newblock The {M}inakshisundaram-{P}leijel coefficients for the vector-valued
  heat kernel on compact locally symmetric spaces of negative curvature.
\newblock {\em Trans. Amer. Math. Soc.}, 260(1):1--33, 1980.

\bibitem[MP12]{MP}
W.~M{\"u}ller and J.~Pfaff.
\newblock Analytic torsion of complete hyperbolic manifolds of finite volume.
\newblock {\em J. Funct. Anal.}, 263(9):2615--2675, 2012.

\bibitem[MS89]{MS}
H.~Moscovici and R.~J Stanton.
\newblock Eta invariants of {D}irac operators on locally symmetric manifolds.
\newblock {\em Invent. Math.}, 95(3):629--666, 1989.

\bibitem[M{\"u}l11]{Mu}
W.~M{\"u}ller.
\newblock A {S}elberg trace formula for non-unitary twists.
\newblock {\em Internat. Math. Res. Notices}, 2011(9):2068--2109, 2011.

\bibitem[Pfa12]{Pf}
J.~Pfaff.
\newblock {\em Selberg and {R}uelle zeta functions and the relative analytic
  torsion on complete odd-dimensional hyperbolic manifolds of finite volume}.
\newblock PhD thesis, 2012.

\bibitem[Pfa13]{Pf1}
J.~Pfaff.
\newblock Selberg zeta functions on odd-dimensional hyperbolic manifolds of
  finite volume.
\newblock {\em J. Reine Angew. Math.}, 2013.

\bibitem[Sel56]{Sel}
A.~Selberg.
\newblock Harmonic analysis and discontinuous groups in weakly symmetric
  riemannian spaces with applications to {D}irichlet series.
\newblock {\em J. Indian Math. Soc}, 20(956):47--87, 1956.

\bibitem[Shu87]{Shu}
M.A. Shubin.
\newblock {\em Pseudodifferential operators and spectral theory}.
\newblock Springer, 1987.

\bibitem[{Spi}15]{Spil}
P.~{Spilioti}.
\newblock {Ruelle and Selberg zeta functions for non-unitary twists}.
\newblock {\em arXiv:1506.04672}, June 2015.

\bibitem[SW73]{SW}
P.~J. Sally and G.~Warner.
\newblock The {F}ourier transform on semisimple {L}ie groups of real rank one.
\newblock {\em Acta Math.}, 131(1):1--26, 1973.

\bibitem[Tsu97]{Tsu}
M.~Tsuzuki.
\newblock Elliptic factors of {S}elberg zeta functions.
\newblock {\em Duke Math. J.}, 88(1):29--75, 05 1997.

\bibitem[Wal93]{Wal}
N.~Wallach.
\newblock On the {S}elberg trace formula in the case of compact quotient.
\newblock {\em Representation Theory and Automorphic Forms}, 2:283, 1993.

\end{thebibliography}
\bibliographystyle{alpha}

\end{document}